\documentclass{article}

\usepackage[utf8]{inputenc}
\usepackage{amssymb}
\usepackage{amsthm}
\usepackage{mathtools}
\usepackage{geometry}
\usepackage{hyperref}
\usepackage{authblk} 

\newtheorem{thm}{Theorem}
\newtheorem{lemma}{Lemma}
\newtheorem{corollary}{Corollary}				
\newtheorem{proposition}{Proposition}
\theoremstyle{definition}
\newtheorem{definition}{Definition}
\newtheorem{assumption}{Assumption}
\newtheorem{remark}{Remark}

\usepackage{xcolor}

\def\E{{\mathbb E}}

\def\P{{\mathbb P}}
\def\bR{{\mathbb R}}

\def\P{{\mathbb P}}

\def\cF {\mathcal{F}}
\def\cC {\mathcal{C}}
\def\cD {\mathcal{D}}
\def\cS {\mathcal{S}}
\def\cN {\mathcal{N}}

\def\cB{\mathcal{B}}
\def\mCov {\mbox{Cov}}

\title{Fluctuations for Spatially Extended Hawkes Processes}
\author[1]{Julien Chevallier \thanks{ julien.chevallier1@univ-grenoble-alpes.fr}} 
\author[2]{Guilherme Ost \thanks{guilhermeost@im.ufrj.br}}
\affil[1]{Univ. Grenoble Alpes, CNRS, Grenoble INP\footnote{Institute of Engineering Univ. Grenoble Alpes}, LJK, 38000 Grenoble, France}
\affil[2]{Universidade Federal do Rio de Janeiro, Rio de Janeiro, Brazil}
\date{\today}

\begin{document}

\maketitle

\begin{abstract}
In a previous paper \cite{chevallier2018spatial}, it has been shown that the mean-field limit of spatially extended Hawkes processes is characterized as the unique solution $u(t,x)$ of a neural field equation (NFE). The value $u(t,x)$ represents the membrane potential at time $t$ of a typical neuron located in position $x$, embedded in an infinite network of neurons.
In the present paper, we complement this result by studying the fluctuations of such a stochastic system around its mean field limit $u(t,x)$. Our first main result is a central limit theorem stating that the spatial distribution associated to these fluctuations converges to the unique solution of some stochastic differential equation driven by a Gaussian noise. In our second main result we show that the solutions of this stochastic differential equation can be well approximated by a stochastic version of the neural field equation satisfied by $u(t,x)$. To the best of our knowledge, this result appears to be new in the literature.
\end{abstract}


\section{Introduction}

We consider multivariate point processes $(N^1,\ldots, N^n)$ on $[0,\infty)$ representing the time occurrences of action potentials (often called \emph{spikes}) of a network of $n$ neurons. We assume that the intensity process of $N^i$ is of the form
\begin{equation}
\label{def:SEHP_intensity}
\begin{cases}
\lambda^i_t=f\left(U^i_{t-}\right),\\
U^i_t=e^{-\alpha t}u_0(x_i)+\frac{1}{n}\sum_{j=1}^nw(x_j,x_i)\int_{0}^{t}e^{-\alpha(t-s)}dN^j_s.
\end{cases}
\end{equation}
In the above formula, $U^i_{t}$ describes the membrane potential of neuron $i$ at time $t\geq 0$ and $x_i=i/n$ represents the position of neuron $i$ in the network. The function $f:\bR\to\bR_+$ is the \emph{firing rate} of each neuron. The function $w:[0,1]\times[0,1]\to\bR$ is the matrix of \emph{synaptic strengths}. It introduces a spatial structure in the model; the value $w(x_j,x_i)$ models the influence of a spike of neuron $j$ on neuron $i$, as a function of their positions. On the one hand, when the sign of $w(x_j,x_i)$ is positive, neuron $j$ excites neuron $i$. On the other hand, if the sign of $w(x_j,x_i)$ is negative, neuron $j$ inhibits neuron $i$.  
The \emph{leakage rate} is modeled by the parameter $\alpha\geq 0.$  The function 
$u_0:[0,1]\to\bR$ describes the membrane potential of all neurons in the network at time $t=0$. We refer to $f,w,u_0$ and $\alpha$ as parameters of the multivariate point process $(N^1,\ldots, N^n).$

Such point processes are known as nonlinear {\it Hawkes Processes}, named after the pioneer work of A. G. Hawkes \cite{hawkes_1971} where the model has been introduced in the linear case (i.e., for $f$ linear). Their defining characteristic is that past events (spikes in our framework) can affect the
probability of future events to occur. 
The literature of neuronal modeling via Hawkes processes is vast. To cite just a few articles, see for instance \cite{Witten_et_al:17, chevallier2015microscopic,Chornoboy:88, SusanneEva:17, hansen2015lasso, hodara_locherbach:17, Johnson:96, Pernice:11, Bouret:14} and the references therein.

Recently, in \cite{chevallier2018spatial}, the authors have established a connection between solutions of (scalar) {\it neural field equations}(NFE) and mean field limits of nonlinear Hawkes processes. 
Specifically, it has been proved that the multivariate process $(U^1_t,\ldots, U^n_t)_t$ defined in \eqref{def:SEHP_intensity} converges as $n\to\infty$, under some assumptions on the parameters of the model,  to a deterministic
function $u(t,x)$ which solves the neural field equation:
\begin{equation}
\label{def:neural_field_equation}
\begin{cases}
\dfrac{\partial}{\partial t}u(t,x)=-\alpha u(t,x)+ \int_{0}^{1} w(y,x)f(u(t,y))dy, \ t>0 \ \mbox{and} \ x\in [0,1],\\
u(0,x)= u_{0}(x).
\end{cases}
\end{equation}
Here, $u(t,x)$ represents the membrane potential at time $t$ of a typical neuron located in position $x$, embedded in an infinite network of neurons.
Neural field equations have been widely studied in the literature since the pioneer works of Wilson, Conwan \cite{wilson_Conwan:72, Wilson_Conwan:73} and Amari \cite{Amari:77} in the 1970s. Such models have attracted a great interest from the scientific community, due to its wide range of applications and mathematical tractability;  see \cite{bressloff2011spatiotemporal} for a recent and comprehensive review. 

The goal of the present paper is to complement the results in \cite{chevallier2018spatial} by describing the fluctuations of the process $(U^1_t,\ldots, U^n_t)_t$ around its mean field limit $u(t,x).$ More precisely, by writing $\eta^i_t=n^{1/2}(U^i_t-u(t,x_i))$ to denote the \emph{individual fluctuations}, the purpose of this paper is to study the convergence of the sequence of stochastic processes $(\Gamma^n_t)_t$ as $n\to\infty$, where $\Gamma^n_t$ is the random signed measure on $\cS'$ (representing the \emph{spatial fluctuations}) defined as
\begin{equation}
\label{def:gamma_n_t}
\Gamma^n_t(dx)=\frac{1}{n}\sum_{i=1}^n\eta^i_t\delta_{x_i}(dx).
\end{equation}
Here, the set $\cS'$ denotes the dual space of the Fréchet space $S=\cC^{\infty}([0,1])$, the space of all real-valued functions on $[0,1]$ with continuous derivatives of all orders.  
Fix $T\geq 0$, denote $\Gamma^n=(\Gamma^n_t)_{0\leq t\leq T}$ and observe that $\Gamma^n\in \cD([0,T],\cS'),$
the space of càdlàg functions from $[0,T]$ to $\cS'$.
Our first main result, namely Theorem \ref{thm:main_result}, is a Central Limit Theorem saying that under some assumptions on the parameters of the model, the sequence of processes $(\Gamma^n)_{n\geq 1}$ converges in law to a limit process $\Gamma=(\Gamma_t)_{0\leq t\leq T}$ as $n\to\infty$. Moreover, the limit process $\Gamma$ belongs to $\cC([0,T],\cS')$, the set of continuous functions from $[0,T]$ to $\cS'$, and for each $t\geq 0$, the measure $\Gamma_t\in \cS'$ is characterized by the following identity: for all $\varphi\in \mathcal{S}$,
\begin{equation}
\label{def:limit_process}    
\Gamma_t(\varphi)=e^{-\alpha t}M_t(\varphi)+\int_{0}^{t}e^{-\alpha(t-s)}\Gamma_s\left(\int_{0}^1\varphi(x)w(\cdot,x)f'(u(s,\cdot))dx\right)ds,
\end{equation}
where $M=(M_t)_{t\geq 0}$ is a continuous centered Gaussian process taking values in $\cS'$ with covariance function given, for all $t_1,t_2\geq 0$ and $\varphi_1,\varphi_2\in \cS$, by
\begin{equation}
\label{eq:covariance:of:M}
    \begin{cases}
    \E(M_{t_1}(\varphi_1)M_{t_2}(\varphi_2))=\int_0^{t_1\wedge t_2} \int_{0}^1 e^{2\alpha s} I[\varphi_1](y) I[\varphi_2](y)     f(u(s,y)) dy ds,\\
    I[\varphi](y)=\int_{0}^1\varphi(x)w(y,x)dx, \quad y\in [0,1],
    \end{cases}
\end{equation}
and $u(t,x)$ is the solution of \eqref{def:neural_field_equation}. The interested reader is referred to \cite[$\Phi'$-Wiener processes]{kallianpur2013stochastic} for details on such Gaussian processes. 

Let us give some intuition about Equation \eqref{def:limit_process}. The first term in the RHS of \eqref{def:limit_process}, namely $e^{-\alpha t}M_t(\varphi)$, comes from the error one makes when replacing the point measure $dN^i_t$ by the intensity measure $f(U^i_t)dt$. It is the diffusion approximation for point processes: formally taking $\varphi_1=\varphi_2=\delta_x$ the Dirac mass at position $x$, one obtains in Equation \eqref{eq:covariance:of:M} the product $w(y,x)^2 f(u(s,y))$ which is the limit variance of the jumps induced by spiking neurons in position $y$ onto neurons in position $x$, at time $s$. 
The second term in the RHS of \eqref{def:limit_process} comes from the error one makes when replacing the intensity $f(U^i_t)$ by the limit one $f(u(t,x_i))$: the linearization of $f$ gives the product of the derivative $f'$ times the difference between $U^i_t$ and $u(t,x_i)$ (which is encapsulated in $\eta^i_t$ and so in the spatial fluctuation $\Gamma^n_t$). 

The study of the fluctuations is a natural follow-up to the study of the mean-field limits for interacting particle systems (see for instance \cite{Braun1977vlasov,chevallier2017fluctuations,DAWSON1991law,delarue2019masterequation,HITSUDA1986tightness,kurtz2004stochastic,luccon2016transition,Meleard_96,tran2006modeles}). 
These results are not only interesting per se, they are also relevant from an applied point of view. Indeed, in the mean-field limit, typically one can show that the so-called {\it propagation of chaos property} holds, meaning that evolution of any finite number of particles (the neurons in our framework) become independent (see for instance \cite{Baladron:12,bossy2015clarification}). In other terms, mean field limits neglect the correlations between particles which are present in finite (but large) systems.
In contrast, the correlations do appear in the fluctuations, in particular in the covariance kernel \eqref{eq:covariance:of:M}.

With slight abuse of terminology, the mean field limit $u_t=u(t,\cdot)$, which can be seen as an element of $\mathcal{S}'$ given by $u_t(\varphi)=\int_0^1 \varphi(x)u(t,x)dx$, can be thought of as a \emph{zeroth-order approximation} of the finite size system $(U^1_t,\ldots, U^n_t)_t$.
In that respect, we say that the following process with values in $\mathcal{S}'$,
\begin{equation}\label{eq:def:second:order:approx}
    (u_t + n^{-1/2}\Gamma_t)_t,
\end{equation}
is a \emph{first-order approximation} of the finite size system, this last definition being justified by our Central Limit Theorem. In addition to the Central Limit Theorem, we also investigate here the link between the first-order approximation and the solution of the following stochastic neural field equation
\begin{equation}
\label{eq:snfe:intro}
\begin{cases}
dV^n_t(x)=\left( -\alpha V^n_t(x)+ \int_{0}^1 w(y,x) f(V^n_t(y))dy \right) dt + \int_{0}^1 w(y,x) \frac{\sqrt{f(V^n_t(y))}}{\sqrt{n}} W(dt,dy),\\
V^n_0(x)=u_0(x),
\end{cases}
\end{equation}
where $W$ is a {\it Gaussian white noise} on $\bR_+\times [0,1]$. 
Loosely speaking, in our second main result, namely Theorem \ref{thm:approximation:snfe}, we show that the process $(u_t + n^{-1/2}\Gamma_t)_t$ is an ``almost'' solution of \eqref{eq:snfe:intro}. 
To the best of our knowledge, this result appears to be new in the literature
and is of independent interest. To some extent, the solutions of \eqref{eq:snfe:intro} can be interpreted as an intermediate modeling scale, sometimes called \emph{mesoscopic scale}, between the microscopic scale given by Hawkes process   \eqref{def:SEHP_intensity} and the macroscopic scale one given by neural field equation \eqref{def:neural_field_equation}.
In order to give sense to solutions of \eqref{eq:snfe:intro} we follow the approach developed by Walsh (see for instance \cite{faugeras2015stochastic}, \cite{Dalang2009AMO} and the seminal lecture notes \cite{walsh1986introduction}). Some heuristics arguments leading to the stochastic neural field equation \eqref{eq:snfe:intro} are provided in Section \ref{sec:heuristics}.
Let us mention the article \cite{chevallier2017fluctuations} which discusses similar results in a non rigorous way in the context of non linear stochastic partial differential equations.

The literature devoted to mean-field limits is usually concerned with the convergence of an empirical measure towards a probability measure which is characterized as the solution of some partial differential equation. It is worth mentioning that it is not the case here: the mean-field equation \eqref{def:neural_field_equation} is not satisfied by a probability density of the potential but by the value of the potential itself. This difference makes the study of \eqref{eq:snfe:intro} simpler: the square root term, namely $\sqrt{f(V_t(y))}$, is trivially well-defined which is not the case when the mean field limit concerns an empirical measure (see \cite{chevallier2017fluctuations} for instance).

The results of the present paper are stated in the distribution space $\cS'$ so the parameters of the model ($f$, $w$ and $u_0$) are assumed to be smooth. 
Concerning the rate function $f$, we also assume that its first and second derivatives are bounded (in particular, $f$ is Lipschitz) and that it is lower-bounded by a positive constant (only in the last section). No additional assumptions on the model are needed and, in particular, the function $f$ could be unbounded.

The present paper is organized as follows. In Section \ref{Sec:2}, the notation used throughout the paper is introduced, the model is described and our first main result, Theorem \ref{thm:main_result}, is stated. In Section \ref{sec:solutions_NFE}, some regularity properties of solutions of the neural field equation are derived.
Uniform estimates on the second moment of the individual fluctuations (used all along the paper) are provided in Section \ref{sec:estimates}. Section \ref{sec:tightness} is devoted to the proof of the tightness of the sequence $(\Gamma^n)_n$ defined in \eqref{def:gamma_n_t}. In Section \ref{sec:limit_equation}, we show that the limit of any converging sub-sequence of $(\Gamma^n)_n$ solves the limit equation \eqref{def:limit_process}. 
In Section \ref{sec:convergence}, the uniqueness of solutions of the limit equation \eqref{def:limit_process} is proved which concludes the proof of the Central Limit Theorem (Theorem \ref{thm:main_result}).
In Section \ref{sec:conec_stoc_NFE}, we first develop the mathematical framework required to study the stochastic neural field equation \eqref{eq:snfe:intro} and then we prove our second main result, Theorem \ref{thm:approximation:snfe}, which makes the link between the first-order approximation \eqref{eq:def:second:order:approx} and the stochastic neural field equation \eqref{eq:snfe:intro}.
Some technical results used in the previous sections are collected in the \ref{App:lemmas}. We include in \ref{app:frechet} some basic definitions about Fréchet spaces.

\section{General notation, model definition and the central limit theorem}
\label{Sec:2}

\subsection{General notation}
\label{subsec:general_notation}
Let $E$ and $F$ be some metric spaces. The space of continuous (respectively càdlàg) functions from $E$ to $F$ is denoted by $\cC(E,F)$ (resp. $\cD(E,F)$). When $F=\bR$, we write $\cC(E)$ (resp. $\cD(E)$) instead of $\cC(E,\bR)$ (resp. $\cD(E,\bR)$). For each integer $n\geq 1,$ let $[n]=\{1,\ldots,n\}$. We write $\cC^{\infty}([0,1])$ (resp. $\cC^{\infty}(\bR)$) to denote the set of all functions $\varphi:[0,1]\to \bR$ (resp. $\varphi:\bR\to \bR$ ) with continuous derivatives of all orders. Similarly, we write $\cC^{\infty}([0,1]\times [0,1])$ to denote the set of all functions $\psi:[0,1]\times [0,1]\to \bR$ with continuous partial derivatives of all orders.   
To ease the notation, the partial derivatives with respect to the first and second variable of a differentiable function $\psi:[0,1]\times [0,1]$ are respectively denoted by $\partial_1 \psi$ and $\partial_2 \psi$. 
Throughout the paper, the letter $C$ denotes a positive constant. Most of the time, the dependence of $C$ with respect to the parameters of the model is specified.

We equip the space $\mathcal{C}^0([0,1])=\mathcal{C}([0,1])$ with the sup norm $$||f||_0=\sup_{x\in [0,1]} |f(x)|.$$
The space $\mathcal{C}^k([0,1])$ of functions with continuous derivatives up to order $k$ is equipped with the norm
\begin{equation}
    ||f||_k = \sum_{i=0}^k ||f^{(i)}||_0,
\end{equation}
where $f^{(0)}=f$ and $f^{(i)}$ denotes the $i$-th derivative of $f$ for $i\in [k]$. 
The space $\mathcal{S}=\cC^\infty([0,1])$ is a Fréchet space \cite{simon2017banach} with the filtering family of semi-norms $(||f||_k)_{k\geq 0}$. Hence it is equipped with the metric $d_\cS$ defined for all $f,g$ in $\mathcal{S}$ by,  
\begin{equation}\label{eq:def:distance:S}
    d_\cS(f,g) := \sum_{k\geq 0} 2^{-k} \frac{|| f - g ||_k}{1 + || f - g ||_k}.
\end{equation}
For a reader not familiar with these notions, some details about Fréchet spaces are gathered in \ref{app:frechet}.

Let $N$ be a point process in $[0,\infty)$, defined on  a filtered probability space $(\Omega, \cF, (\cF_t)_{t\geq 0}, \P)$. 
We say that $N$ is {\it locally finite} if for all $t\geq 0$, the random variable $N_t=N((0,t])$ counting the number of points of $N$ in the interval $(0,t]$ is finite almost surely.
We say that the $(\cF_t)_{t\geq 0}$-predictable process $(\lambda_t)_{t\geq 0}$ is the intensity process of $N$ if the process $(N_t-\int_{0}^t\lambda_s ds)_{t\geq 0}$ is a $(\cF_t)_{t\geq 0}$-local martingale. For bounded measurable functions $g:[0,\infty)\to \bR$ and a locally finite point process $N$, we define $\int_{0}^{t}g(s)dN_s=\sum_{s\in N\cap (0,t]}g(s)$ for any $t>0.$

For any locally square integrable martingale $(M_t)_{t\geq 0}$, the Doob-Meyer decomposition gives rise to the angle bracket, usually denoted by $(\langle M \rangle_t)_{t\geq 0}$, which is the unique non-decreasing predictable process such that $\langle M \rangle_0=0$ and $(M_t^2- \langle M \rangle_t)_{t\geq 0}$ is local martingale.

\subsection{Model definition and the central limit theorem}

Throughout the paper we work on a filtered probability space $(\Omega, \cF, (\cF_t)_{t\geq 0}, \P)$. 
 We assume that this filtered probability space is rich enough so that all the processes we shall consider may be defined on it. 
 We consider a nonlinear Hawkes process $(N^{1},\ldots N^{n})$ in $[0,\infty)$ representing the spiking activity of $n$ interacting neurons. We assume that neuron $i\in [n]$ is located at position $x_i=i/n.$ The dynamics of the Hawkes process $(N^1,\ldots, N^n)$ is described as follows.

\begin{definition}
\label{def:SEHP}
Let $f:\bR\to\bR_+$, $w:[0,1]\times [0,1]\to\bR$ and $u_0:[0,1]\to\bR$ be measurable functions and $\alpha\geq 0$ be a fixed parameter.  We say $(N^1,\ldots, N^n)$ is a Hawkes process with parameters $(f,w,u_0,\alpha)$ if
\begin{enumerate}
    \item $\P-$ almost surely, for all pairs $i,j\in [n]$ with $i\neq j$, the point processes $N^i$ and $N^j$ never jump simultaneously.
    \item For each $i\in [n]$, the intensity process $(\lambda^i_t)_{t\geq 0}$ of $N^i$ is given by $\lambda^i_t=f(U^i_{t-})$, where $U^i_{t}$ is defined by
\begin{equation}
\label{def:membrane_potential}
U^i_{t}=e^{-\alpha t}u_0(x_i)+\frac{1}{n}\sum_{j=1}^nw(x_j,x_i)\int_{0}^te^{-\alpha(t-s)}dN^j_s.    
\end{equation}
\end{enumerate}
\end{definition}

We shall work under the following assumption on the parameters $(f,w,u_0,\alpha)$ of the model. 

\begin{assumption}
\label{ass:f}
The function $f:\bR\to\bR_+$ belongs to $\cC^{\infty}(\bR)$.
Moreover, the first  and second derivatives of $f$ are both bounded, that is  $\|f'\|_0<\infty$ and $\|f^{(2)}\|_0<\infty$. 
Furthermore, the functions $u_0:[0,1]\to\bR$ and $w:[0,1]\times [0,1]\to \bR$ are both smooth, that is, $u_0\in \cC^{\infty}([0,1])$ and $w\in \cC^\infty([0,1]\times [0,1])$.
\end{assumption}
Note that under the assumption $\|f'\|_0<\infty$,  the function $f$ is Lipschitz continuous. 

\begin{remark}
Here we briefly discuss some examples of functions $f$, $w$ and $u_0$ satisfying Assumption \ref{ass:f}. They are widely used in the literature (see the reviews \cite{bressloff2011spatiotemporal,ermentrout1998neural} for instance).
\begin{itemize}
    \item \textbf{firing rate $f$:} the sigmoid rate $f(u)=f_0/(1+e^{-(u-\kappa)})$ and the gaussian rate $f(u)=f_0(1+\operatorname{erf}(u-\kappa))/2$, with $\operatorname{erf}(x)=(2/\sqrt{\pi})\int_0^x e^{-t^2/2}dt$, where $\kappa\in \mathbb{R}$ can be thought as a threshold and $f_0>0$ a maximal firing capacity of the neurons; 
    \item \textbf{synaptic strength $w$:} the standard form is $w(x,y)=w(|x-y|)$. In that framework, the gaussian $w(x)=e^{-x^2}$, the exponential $w(x)=e^{-x}$ and the mexican hat function are widely used. The latter writes as the difference of two gaussians or two exponentials: for instance $w(x)=e^{-x^2} - Ae^{-x^2/\sigma}$ with $A<1$ and $\sigma>1$ describes short range excitation and long range inhibition;
    \item \textbf{initial condition $u_0$:} a constant function or a smooth interpolation between $u_0=0$ and $u_0=a>0$.
\end{itemize}
\end{remark}

\begin{remark}
Throughout the paper we work with smooth functions. 
We do this partly in order to avoid some technicalities which make our proofs less transparent.
Following the approach adopted in \cite{chevallier2017fluctuations,Meleard_96,luccon2016transition} it is possible to state the central limit theorem in some Hilbert space and  weaken the assumptions to consider functions that are only twice differentiable with bounded derivatives. 
\end{remark}
\begin{remark}
Note that the positions $x_i$'s are regularly spaced in the compact set $[0,1].$ We stress that our results do not rely on this specific choice. They can be easily extended to the case in which the positions $x_i$ belong to a regular grid of some compact set $K\subset \bR^d$ for some integer $d\geq 1$, at the cost of more complicated notation.
\end{remark}

For each $t\geq 0$, let $U_t=(U^1_t,\ldots, U^n_t)$. Under some assumptions on the functions $f,w$ and $u_0$ (much weaker than those of Assumption \ref{ass:f}), 
it has been proved \cite[Corollary 2]{chevallier2018spatial} that the process $(U_t)_{t\geq 0}$ converges (in some sense) to the unique solution $u(t,x)$ of the scalar neural field equation \eqref{def:neural_field_equation}.

Recall that we write $\cS$ to denote $\cC^{\infty}([0,1])$ and $\cS'$ to denote its dual space.  
The main goal of this paper is to describe the fluctuations of $(U_t)_{t\geq 0}$ around its limit, the continuous deterministic solution $u(t,x)$ of the neural field equation \eqref{def:neural_field_equation}. For this reason, for each $i\in [n]$ and $t\geq 0$, we define the individual fluctuations
$\eta^i_t=n^{1/2}(U^i_t-u(t,x_i))$ 
and consider the random signed measures $\Gamma^n_t$ on $\cS'$ (representing the spatial fluctuations) defined as
$$
\Gamma^n_t(dx)=\frac{1}{n}\sum_{i=1}^n\eta^i_t\delta_{x_i}(dx).
$$
For some fixed $T>0$, denote $\Gamma^n=(\Gamma^n_t)_{0\leq t\leq T}$ and observe that $\Gamma^n\in \cD([0,T],\cS').$
Our first main result is the following.

\begin{thm}
\label{thm:main_result}
Under Assumption \ref{ass:f}, the sequence $(\Gamma^n)_{n\geq 1}$ converges in law in $\mathcal{D}([0,T], \mathcal{S}')$ to the unique solution $\Gamma=(\Gamma_t)_{0\leq t\leq T}\in \mathcal{C}([0,T], \mathcal{S}')$ of  equation \eqref{def:limit_process}.
\end{thm}

The proof of Theorem \ref{thm:main_result} is divided in several steps.  
We first derive some regularity properties of solutions of the NFE \eqref{def:neural_field_equation} - see Proposition \ref{prop:contraction:NFE} (its proof is based mainly on results provided in \cite{chevallier2018spatial}). 
Next we  prove tightness of the sequence $(\Gamma^n)_{n\geq 1}$ in $\cD([0,T],\cS')$. To that end, we rely on \cite[Theorem 4.1]{Mitoma_89}, according to which the tightness of the sequence $(\Gamma^n)_{n\geq 1}$ in $\cD([0,T],\cS')$ follows from the tightness of the sequence $(\Gamma^n(\varphi))_{n\geq 1}$ in $\cD([0,T],\bR)$ for each $\varphi\in\cS$, where $\Gamma^n(\varphi)=(\Gamma^n_t(\varphi))_{0\leq t\leq T}$ and for each $0\leq t\leq T$,
\begin{equation*}
\Gamma^n_t(\varphi)=\frac{1}{n}\sum_{i=1}^n\eta^i_t\varphi(x_i).
\end{equation*}
To show the tightness of $(\Gamma^n(\varphi))_{n\geq 1}$ in $\cD([0,T],\bR)$, we first decompose $\Gamma^n_t(\varphi)$ as 
\begin{equation}
\Gamma^n_t(\varphi)=e^{-\alpha t}M^n_t(\varphi)+B^n_t(\varphi)+C^{n}_t(\varphi),    
\end{equation}
where $M^n(\varphi)=(M^n_t(\varphi))_{0\leq t\leq T}$ is a local martingale, $B^n(\varphi)=(B^n_t(\varphi))_{0\leq t\leq T}$ is a continuous stochastic process and $C^n(\varphi)=(C^{n}_t(\varphi))_{0\leq t\leq T}$ is a continuous function: all these quantities are carefully defined in Equation \eqref{def:A_B_C_phi}. We then show (see Proposition \ref{prop:tightness_of_A_B_M:in:R}) that the sequence of functions $(C^n(\varphi))_{n\geq 1}$ goes to $0$ and use {\it Aldous criterion} to show that both sequences  $(M^n(\varphi))_{n\geq 1}$ and $(B^n(\varphi))_{n\geq 1}$ are tight.  From that it is easy to conclude
the tightness of $(\Gamma^n(\varphi))_{n\geq 1}$ in $\cD([0,T],\bR)$ - see Corollary \ref{cor:tightness_of_Gamma_M_in_Sprime}.

Once established the tightness of the sequence $(\Gamma^n)_{n\geq 1}$, we show that its limit points belong to $\cC([0,T],\bR)$ and satisfy equation \eqref{def:limit_process} - see Proposition \ref{prop:continuity:limit:Gamma} and Theorem \ref{thm:limit:is:solution:of:limit:equation} respectively.
To conclude the proof of Theorem \ref{thm:main_result}, we then prove that solutions of equation \eqref{def:limit_process} are unique - see Theorem \ref{thm:CLT}.

\section{Solutions of the Neural Field Equation}
\label{sec:solutions_NFE}

The purpose of this section is to show regularity properties for the solution $u(t,x)$ of the NFE involved in the definition of the individual fluctuations $(\eta^i_t)_{0\leq t\leq T}$. In the preliminary study made in \cite{chevallier2018spatial}, some regularity properties of $u(t,x)$ are shown. Then, using this a priori regularity we are able to show that $u(t,x)$ is in fact smooth.

In \cite{chevallier2018spatial}, the function of interest is not the limit potential $u(t,x)$ but the limit intensity $\lambda(t,x)$ which is proven to be continuous and uniquely characterized as the unique \emph{physical solution}\footnote{By physical solution, we mean a solution which satisfies some a priori property inherited from the microscopic model (see  \cite[equation above Proposition 5]{chevallier2018spatial})} of some fixed point equation. Nevertheless these two functions are closely linked by \cite[Equation (3.20)]{chevallier2018spatial}:
\begin{equation}\label{eq:system:u:lambda}
\begin{cases}
    u(t,x) = e^{-\alpha t} u_0 (x) + \int_0^t e^{-\alpha (t-s)} \int_{0}^{1} w(y,x) \lambda(s,y) dy ds, \ t>0 \ \mbox{and} \ x\in [0,1],\\
    \lambda(t,x)=f(u(t,x)).
\end{cases}
\end{equation}
Since $\lambda(t,x)$ belongs to $\cC([0,T]\times [0,1],\bR_+)$ \cite[Proposition 5]{chevallier2018spatial}, it then follows that $u(t,x)$ belongs to $\cC([0,T]\times [0,1],\bR)$ (which can be identified to $\cC([0,T], \cC([0,1])))$. In the following, the evaluation of a function $u(t,x)\in \cC([0,T], \cC([0,1]))$ is rather denoted by $u_t(x)$.

In particular, Equation \eqref{eq:system:u:lambda} means that $u_t(x)$ is a fixed point of the map $F$ defined by : for all $v\in \mathcal{C}([0,T],\mathcal{C}([0,1]))$, for all $t\geq 0$ and $x\in [0,1]$,
\begin{equation}\label{eq:def:contraction:NFE}
     F(v)_t(x) := e^{-\alpha t} u_0 (x) + \int_0^t e^{-\alpha (t-s)} \int_{0}^1 w(y,x) f(v_{s}(y)) dy ds,
\end{equation}
where $u_0$ is the initial condition.

\begin{proposition}\label{prop:contraction:NFE}
Under Assumption \ref{ass:f}, for all $v\in \mathcal{C}([0,T],\mathcal{C}([0,1]))$, $F(v)$ belongs to the smaller space $\mathcal{C}([0,T],\mathcal{S})$. In particular, there is a unique physical solution of the NFE and existence of a smooth solution.
\end{proposition}

\begin{proof}
Let $v$ be in $\mathcal{C}([0,T],\mathcal{C}([0,1]))$. In particular, $v$ is locally bounded in time ($\sup_{t\leq T, x\in [0,1]} v_t(x)<+\infty$) so, using the Lipschitz continuity of $f$ and the smoothness of $w$ and $u_0$, it is clear that for all $t\leq T$, $F(v)_t\in \mathcal{S}$ and that
\begin{equation*}
    F(v)_t^{(k)}(x) = e^{-\alpha t} u_0^{(k)}(x) + \int_0^t e^{-\alpha (t-s)} \int_{0}^1 \partial_2^k w(y,x) f(v_{s}(y)) dy ds.
\end{equation*}
Let $s\leq t\leq T$, using the Lipschitz continuity of $f$ and the fact that $||\partial_2^k w(y,x)||_0<\infty$, we have for all $k\geq 1$,
\begin{eqnarray*}
||(F(v)_t)^{(k)} - (F(v)_s)^{(k)} ||_0 &\leq& e^{-\alpha s} |e^{\alpha (t-s)} - 1|\, ||u_0^{(k)}||_0 \\
&&+ \int_s^t \left| \int_{0}^1  \partial_2^k w(y,x) f(v_{h}(y)) dy\right| dh\\
&& + \int_0^s e^{\alpha h} \left| e^{-\alpha t} - e^{-\alpha s} \right| \left| \int_{0}^1  \partial_2^k w(y,x) f(v_{h}(y)) dy\right| dh\\
&\leq & C(t-s) \, ||u_0^{(k)}||_0 + CT(t-s)e^{\alpha T} (1+ \sup_{t\leq T} ||v_t||_0).
\end{eqnarray*}
Summing up, we get $||F(v)_t - F(v)_s ||_k\leq kC_T (t-s) (1 + ||u_0||_k + \sup_{t\leq T} ||v_t||_0)$.

Let $\varepsilon>0$ and $k_0$ be such that $\sum_{k=k_0+1}^{+\infty} 2^{-k}<\varepsilon$. It suffices then to take $s$ and $t$ close enough such that
\begin{equation*}
    k_0 C_T (t-s) (1 + ||u_0||_{k_0} + \sup_{t\leq T} ||v_t||_0) < \varepsilon,
\end{equation*}
to get $d_\infty(F(v)_t,F(v)_s)\leq 2\varepsilon$ and so $F(v)\in \mathcal{C}([0,T],\mathcal{S})$.

Assume that $u$ and $\tilde{u}$ are two physical solutions of the NFE. Then, $\lambda(t,x)=f(u(t,x))$ and $\tilde{\lambda}(t,x)=f(\tilde{u}(t,x))$ define two physical solutions of the fixed point equation \cite[Equation (3.10)]{chevallier2018spatial}. Hence, uniqueness for $\lambda$ proved in \cite{chevallier2018spatial} implies uniqueness for $u$. Existence is already proven in \cite{chevallier2018spatial}.
\end{proof}

\section{First estimates}
\label{sec:estimates}

In the sequel, for each $t\geq 0$ and $i\in[n]$, we write
$$
\begin{cases}
M^i_t=N^i_t-\int_{0}^tf(U^i_s)ds,\\
g(s,x_i)=\frac{1}{n}\sum_{j=1}^n w(x_j,x_i)f(u(s,x_j)).
\end{cases}
$$ 
Recall (see Section \ref{subsec:general_notation}) that $(M^i_t)_{t\geq 0}$ is the local martingale associated with neuron $i$. 
With this notation, by using \eqref{def:membrane_potential} and \eqref{def:neural_field_equation},
we can rewrite $\eta^i_t=n^{1/2}(U^i_t-u(t,x_i))$ as follows:
$$
\eta^i_t=A^i_t+B^i_t+C^i_t,
$$
where $A^i_t$, $B^i_t$ and $C^i_t$ are given respectively by
\begin{equation}
\begin{cases}
\label{def:eta_it_decomposition}
A^i_t=e^{-\alpha t}n^{-1/2}\sum_{j=1}^n \int_{0}^te^{\alpha s}w(x_j,x_i)dM^j_s,\\
B^i_t=n^{-1/2}\sum_{j=1}^n \int_{0}^te^{-\alpha (t-s)}w(x_j,x_i)\left(f(U^j_s))-f(u(s,x_j))\right)ds,\\
C^i_t=n^{1/2}\int_{0}^te^{-\alpha(t-s)}
\left(g(s,x_i)-\int_{0}^1w(y,x_i)f(u(s,y))dy\right)ds.
\end{cases}
\end{equation}
Note that $(C^i_t)_{t\geq 0}$ is deterministic, while both $(A^i_t)_{t\geq 0}$ and $(B^i_t)_{t\geq 0}$ are stochastic. Furthermore, $(A^i_t)_{t\geq 0}$ belongs to $\cD(\bR_+)$ (but is not a local martingale even if $(M^i_t)_{t\geq 0}$ is) and $(B^i_t)_{t\geq 0}$ belongs to $\cC(\bR_+)$. Although every object defined above depends on $n$,  we omit this dependence to ease the notation.

We start this section with the following result.

\begin{proposition}
\label{Prop:upper_bound_eta_it_square}
Assume that $f\in\cC^1(\bR)$ is Lipschitz continuous, $u_0$ is Lipschitz continuous, $u(t,x)\in\cC([0,T],\cC^1([0,1]))$ and $w$ is bounded such that $y\to\partial_1w(y,x)$ exists for all $x\in [0,1]$ and $\sup_{x\in [0,1]}\|\partial_1 w(\cdot, x)\|_{0}<\infty$.  Then, for each $T>0,$
\begin{equation}
  \sup_{n\geq 1} \sup_{0\leq t\leq T} \max_{i\in [n]} \E\left[(\eta^i_t)^2\right]<\infty.      
\end{equation}
\end{proposition}

\begin{proof}
By Jensen inequality, we have that
$$
\E\left[(\eta^i_t)^2\right]\leq 3\left(\E\left[ (A^i_t)^2\right]+\E\left[(B^i_t)^2\right]+(C^i_t)^2\right).
$$
Now, we will bound from above each term on the RHS of the inequality above. 
We will start with $\E\left[ (A^i_t)^2\right].$ To that end, we use \cite[Proposition II.4.1.]{gill_1997} and the fact that $w$ is bounded to obtain that for all $0\leq t\leq T$, 

\begin{align}
\label{Control_on_Ai_t}
\E\left[ (A^i_t)^2\right]&=e^{-2\alpha t}\frac{1}{n}\sum_{j=1}^n\E\left[\int_{0}^te^{2\alpha s}w^2(x_j,x_i)f(U^j_s)ds\right] \nonumber\\
&\leq \|w\|_0^2\frac{1}{n}\sum_{j=1}^n\E\left[\int_{0}^tf(U^j_s)ds\right] \nonumber\\
&=\|w\|_0^2\frac{1}{n}\sum_{j=1}^n\E\left[N^j_t\right]\leq \|w\|_0^2 \frac{1}{n}\sum_{j=1}^n\E\left[N^j_T\right].
\end{align}
Since $w$ is bounded and $u_0$ is Lipschitz continuous on $[0,1]$  (hence bounded as well),  \cite[Proposition 3]{chevallier2018spatial} implies that not only the RHS of \eqref{Control_on_Ai_t} is finite, but also that 
\begin{equation}
\label{Control_on_sup_T_sup_Ai_t}
    \sup_{n\geq 1}\sup_{0\leq t\leq T}\max_{i\in [n]}\E[(A^i_t)^2]<\infty.
\end{equation}

Next, we will deal with the term $(C^i_t)^2.$ From the classical Riemann approximation, we have for each $i\in [n]$ and $s\geq 0$,
\begin{multline}
\left |g(s,x_i)-\int_{0}^1 w(y,x_i)f(u(s,y))dy\right|\leq \frac{1}{2n}\sup_{y\in [0,1]}\left | \partial_1w(y,x_i)f(u(s,y)\right. \\ \left. +w(y,x)f'(u(s,y))\partial_2 u(s,y)\right |,  
\end{multline}
and thus we obtain for $t\geq 0$,
$$
|C^i_t|\leq \frac{1}{2n^{1/2}}\max_{j\in [n]}\sup_{0\leq s\leq t,y\in [0,1]}\left | \partial_1w(y,x_j)f(u(s,y)) +w(y,x)f'(u(s,y))\partial_2 u(s,y)\right |.
$$
As a consequence of the inequality above, we have for all $0\leq t\leq T$,
\begin{align*}
|C^i_t|^2&\leq \frac{1}{4n}\left(\sup_{0\leq s\leq T,x, y\in [0,1]}\left | \partial_1w(y,x)f(u(s,y) +w(y,x)f'(u(s,y))\partial_2 u(s,y)\right |\right)^2.
\end{align*}
Since $M_T=\sup_{0\leq s\leq T}\|u(s,\cdot)\|_{0}<\infty$ and $f$ is Lipschitz continuous, we have that $f$ is locally bounded, implying that $\sup_{s\leq T,y,\in [0,1]}|f(u(s,y))|<\infty$. The assumptions on the functions $u$ and $w$ ensure that both $\sup_{x,y\in [0,1]}\|\partial_1 w(y,x)\|_{0}$ and $\sup_{s\leq T,y\in [0,1]}\|\partial_2 w(s,y)\|_{0}$ are finite, so that 
\begin{equation}
\label{Control_on_Ci_t}
\sup _{n\geq 1}\sup_{t\leq T}\max_{i\in [n]}|C^i_t|^2<\infty.
\end{equation}

It remains to deal with the term $\E\left[(B^i_t)^2\right].$ In what follows, fix an integer $k\geq 1$, and consider $\tau_k=\inf\{0\leq t\leq T :\max_{i\in [n]}|\eta^i_t|\geq k\}$.
By applying Jensen inequality twice and using the fact that $f$ is Lipschitz continuous, we deduce that
\begin{align*}
\E\left[(B^i_{t\wedge \tau_k})^2\right]&\leq \sum_{j=1}^n\E\left[\left(\int_{0}^{t\wedge \tau_k} e^{-\alpha ((t\wedge \tau_k)-s)}w(x_j,x_i)(f(U^j_s)-f(u(s,x_j)))ds\right)^2\right]\\
&\leq  t \|f'\|^2_0\frac{1}{n}\sum_{j=1}^n \int_{0}^tw^2(x_j,x_i)\E\left[(\eta^j_{s\wedge\tau_k})^2\right]ds.    
\end{align*}
Now, since $w$ is a bounded function, it follows then that for any $0\leq t\leq T$,
\begin{align}
\label{Control_on_Bi_t}
\E\left[(B^i_{t\wedge \tau_k})^2\right]&\leq T\|f'\|^2_0\|w\|^2_{0}\int_{0}^t\max_{j\in [n]}\E\left[(\eta^j_{s\wedge \tau_k})^2\right]ds.
\end{align}

Combining \eqref{Control_on_sup_T_sup_Ai_t}, \eqref{Control_on_Bi_t} and \eqref{Control_on_Ci_t}, we have that there exists a finite positive constant $C=C(T,w,f,u)$ such that for all $0\leq t\leq T$ and any integer $n\geq 1$,
$$
\max_{i\in [n]}\E\left[(\eta^i_{t\wedge\tau_k})^2\right]\leq C \left( 1+\int_{(0,t]}\max_{i\in [n]}\E\left[(\eta^i_{s\wedge\tau_k})^2\right]ds\right). 
$$
Since $t\to \E[(\eta^i_{t\wedge\tau_k})^2]$ is locally bounded, we may apply Gronwall inequality to conclude that for all $0\leq t\leq T$ and any integer $n\geq 1$
$$
\max_{i\in [n]}\E\left[(\eta^i_{t\wedge\tau_k})^2\right]<C,
$$
for some finite positive constant $C=C(T,w,f,u).$ By Lemma \ref{lemma:truncation}, which is in \ref{App:lemmas}, we know that $\tau_{k}\to T$ a.s. as $k\to\infty$, and hence, by Fatou's lemma, for all $t\leq T$ and integer $n\geq 1$,
$$
\max_{i\in [n]}\E\left[(\eta^i_{t})^2\right]<C,
$$
implying the result.
\end{proof}

\begin{corollary}
Under the assumptions of Proposition \ref{Prop:upper_bound_eta_it_square}, for all bounded functions $\varphi:[0,1]\to \bR$ and $T>0,$
\begin{equation}
\sup_{n\geq 1} \sup_{0\leq t\leq T}\E[(\Gamma^n_t(\varphi))^2]<\infty .   
\end{equation}
\end{corollary}
\begin{proof}
Apply Jensen inequality to deduce that
\begin{align*}
\E((\Gamma^n_t(\varphi))^2)&=\E\left(\left(\frac{1}{n}\sum_{i=1}^n\eta^i_t\varphi(x_i)\right)^2\right)\\
&\leq \frac{1}{n}\sum_{i=1}^n\varphi^2(x_i)\E((\eta^i_t)^2).
\end{align*}
Since $\varphi$ is bounded, the result then follows from Proposition \ref{Prop:upper_bound_eta_it_square}.
\end{proof}

\section{Tightness}
\label{sec:tightness}
The goal of this section is to prove that the sequence of $\cS^{\prime}$-valued stochastic processes $(\Gamma^n)_{n\geq 1}$ is tight in $\cD([0,T],\cS^{\prime}).$ According to Mitoma \cite[Theorem 4.1]{Mitoma_89}, it suffices to show that the sequence of stochastic processes $(\Gamma^n(\varphi))_{n\geq 1}$ is tight in $\cD([0,T],\bR)$, for each fixed  $\varphi\in \cS.$

In what follows, we fix $\varphi\in \cS$ and consider the sequence of stochastic processes $(\Gamma^n(\varphi))_{n\geq 1}.$ Our goal is to show that this sequence is tight in $\cD([0,T],\bR)$. 
To show this, we use {\it Aldous' tightness criterion}. According to Aldous \cite[Theorem 16.10]{Billing_Convergence}, a sequence of stochastic processes $(X^n)_{n\geq 1}$ in $\cD([0,T],\bR)$ is tight if both conditions below are satisfied:
\begin{enumerate}
    \item for any $0\leq t\leq T$ and  $\epsilon>0$, there exist an integer $n_0\geq 1$ and $K>0$ such that
    \begin{equation}
    \label{Def:aldous_condi_1}
    \sup_{n\geq n_0}\P(|X^n_t|>K)\leq \epsilon.    
    \end{equation}
    \item for any $\epsilon_1, \epsilon_2>0$,  there exist $\delta>0$ and an integer $n_0\geq 1$ such that
    \begin{equation}
    \label{Def:aldous_condi_2}
    \sup_{n\geq n_0}\sup_{(\tau',\tau)\in St_{\delta}}\P(|X^n_{\tau'}-X^n_{\tau}|>\epsilon_1)\leq \epsilon_2,    
    \end{equation}
    where $St_{\delta}$ is the set of all pairs $(\tau',\tau)$ of $(\cF_t)_{t\geq 0}$-stopping times such that $\P$-a.s we have $\tau\leq \tau'\leq \tau+\delta\leq T.$  
\end{enumerate}

To verify that $(\Gamma^n(\varphi))_{n\geq 1}$ satisfies Aldous's criterion, it will be convenient to introduce some new notation. 
Note that for each $t\geq 0$ and all $n\geq 1$, the spatial fluctuation $\Gamma_t^n(\varphi)$ can be rewritten as follows:
\begin{equation}
\label{def:gamma_phi_decomposition}
\Gamma^n_t(\varphi)=A^n_t(\varphi)+B^n_t(\varphi)+C^n_t(\varphi),    
\end{equation}
with $A^n_t(\varphi)$, $B^n_t(\varphi)$ and $C^n_t(\varphi)$ given respectively by
\begin{equation}
\begin{cases}
\label{def:A_B_C_phi}
A^n_t(\varphi)=\frac{1}{n}\sum_{i=1}^{n}\varphi(x_i)A^i_t,\\
B^n_t(\varphi)=\frac{1}{n}\sum_{i=1}^{n}\varphi(x_i)B^i_t,\\
C^n_t(\varphi)=\frac{1}{n}\sum_{i=1}^{n}\varphi(x_i)C^i_t.
\end{cases}
\end{equation}
where $A^i_t$, $B^i_t$ and $C^i_t$ for $1\leq i\leq n$ are defined in \eqref{def:eta_it_decomposition}.

For later use, it will be useful to write $A^n_t(\varphi)=e^{-\alpha t} M^n_t(\varphi)$, where $M^n(\varphi)$ is a local martingale  given for all $t\geq 0$ by 
\begin{equation}
\label{def:Martinga_M_n_phi_t}
M^n_t(\varphi)=\frac{1}{n^{1/2}}\sum_{j=1}^n\int_{0}^te^{\alpha s}\frac{1}{n}\sum_{i=1}^nw(x_j,x_i)\varphi(x_i)dM^j_s.
\end{equation}
By writing $I^n[\varphi](y):= n^{-1} \sum_{i=1}^n w(y,x_i)\varphi(x_i)$, one can check (see for instance \cite{gill_1997}) that its angle bracket is given by 
\begin{equation}
\label{def:_angle_bracket_Martinga_M_n_phi_t}
    \langle M^n(\varphi)\rangle_t = n^{-1}\sum_{j=1}^n \int_{0}^t e^{2\alpha s} I^n[\varphi](x_j)^2 f(U^j_s) ds,
\end{equation}
and by the polarization identity,
\begin{equation}
\label{def:_angle_bibracket_Martinga_M_n_phi_t}
    \langle M^n(\varphi_1), M^n(\varphi_2)\rangle_t = n^{-1}\sum_{j=1}^n \int_{0}^t e^{2\alpha s} I^n[\varphi_1](x_j) I^n[\varphi_2](x_j) f(U^j_s) ds.
\end{equation}
With this notation, we can prove the following result.
\begin{proposition}\label{prop:tightness_of_A_B_M:in:R}
Let us make the same assumptions as in Proposition \ref{Prop:upper_bound_eta_it_square}.
Then, for each fixed $\varphi$ in $\cS$, the sequences of the stochastic processes $(M^n(\varphi))_{n\geq 1}$ and $(A^n(\varphi))_{n\geq 1}$ are tight in the space $\mathcal{D}([0,T])$ and $(B^n(\varphi))_{n\geq 1}$ is tight in $\cC([0,T])$. Moreover, the sequence of functions $(C^n(\varphi))_{n\geq 1}$ satisfies
\begin{equation}
\label{control_of_C_phi}
    \sup_{t\in [0,T]} |C^n_t(\varphi)| \leq C n^{-1/2} ||\varphi||_0,
\end{equation}
where $C=C(T,w,f,u)$ is some finite positive constant.
\end{proposition}
\begin{proof}
We need to show that the sequences $(M^n(\varphi))_{n\geq 1}$, $(A^n(\varphi))_{n\geq 1}$ and $(B^n(\varphi))_{n\geq 1}$ satisfy Aldous' criterion. We start with $(M^n(\varphi))_{n\geq 1}$. 
All along the proof, we write $L$ to denote a constant which might change from line to line and depends only on $T,\|w\|_0, \|\varphi\|_0, \alpha$ and  $\|f\|_0.$

{\bf Tightness of $(M^n(\varphi))_{n\geq 1}$.} Fix $0\leq t\leq T$ and $\epsilon>0$. By Markov's inequality, for all $n\geq 1$ and $K>0$, we have
$$
\P(|M^n_t(\varphi)|>K)\leq \frac{1}{K^2}\E(|M^n_t(\varphi)|^2),
$$
so that by \eqref{def:_angle_bracket_Martinga_M_n_phi_t}, we can deduce from the above inequality that
$$
\P(|M^n_t(\varphi)|>K)\leq \frac{1}{K^2}\frac{1}{n}\sum_{j=1}^{n}\E\left(\int_{0}^t\left(n^{-1}\sum_{i=1}^n \varphi(x_i) w(x_j,x_i) \right)^2 f(U^j_s) e^{2\alpha s}ds\right).
$$
Since $w$ and $\varphi$ are bounded, it follows from this last inequality that
$$
\P(|M^n_t(\varphi)|>K)\leq \frac{L}{K^2}\frac{1}{n}\sum_{j=1}^{n}\E\left(\int_{0}^tf(U^j_s)ds\right)\leq \frac{L}{K^2}\sup_{n\geq 1}\frac{1}{n}\left\{\sum_{j=1}^{n}\E\left(N^j_T\right)\right\}.
$$
Then, \cite[Proposition 3]{chevallier2018spatial} together with the inequality above implies that for all $n\geq 1$,
$$
\P(|M^n_t(\varphi)|>K)\leq \frac{L}{K^2},
$$
and the condition \eqref{Def:aldous_condi_1} holds whenever $K\geq \sqrt{L/\epsilon}.$

It remains to show that $(M^n(\varphi))_{n\geq 1}$ satisfies the condition \eqref{Def:aldous_condi_2}. By \cite[Theorem 2.3.2]{Rebolledo_80}, it suffices to show that $(\langle M^n(\varphi)\rangle)_{n\geq 1}$ satisfies the condition \eqref{Def:aldous_condi_2}. Thus, take stopping times $(\tau,\tau')\in St_{\delta}$ and observe that by \eqref{def:_angle_bracket_Martinga_M_n_phi_t}, 
\begin{align*}
\E\left(\left|\langle M^n(\varphi)\rangle_{\tau'}-\langle M^n(\varphi)\rangle_{\tau}\right|\right)&=n^{-1}\sum_{j=1}^n \E\left(\left| \int_{\tau}^{\tau'} \left(n^{-1}\sum_{i=1}^n \varphi(x_i) w(x_j,x_i) \right)^2 f(U^j_s) e^{2\alpha s}ds. \right|\right)\\
&\leq L\frac{1}{n}\sum_{j=1}^n\E\left(\int_{\tau}^{\tau'} f(U^j_s) e^{2\alpha s} ds \right).
\end{align*}
Now, note that for all $j\in [n]$ and $t\geq 0$,
$$
f(U^j_s)\leq f(u(t,x_j))+\frac{\|f'\|_{0}}{\sqrt{n}}|\eta^j_t|\leq \sup_{t\leq T}\|f(u(t,\cdot))\|+\frac{\|f'\|_{0}}{\sqrt{n}}|\eta^i_t|.
$$
The local boundedness of both $f$ and $u$, and the fact that $0\leq \tau\leq\tau'\leq T$, imply then that
$$
\E\left(\left|\langle M^n(\varphi)\rangle_{\tau'}-\langle M^n(\varphi)\rangle_{\tau}\right|\right)\leq L\left(\E(\tau'-\tau)+\frac{1}{n^{3/2}}\sum_{j=1}^n\E\left(\int_{\tau}^{\tau'}e^{2\alpha s}|\eta^j_s|ds\right)\right)
$$
By applying Young inequality, we have for all $s\geq 0$, $j\in [n]$ and $\xi>0$,
$$
e^{2\alpha s}|\eta^j_s|\leq \frac{1}{2\xi}e^{4\alpha s}+\frac{\xi}{2}|\eta^j_s|^2, 
$$
so that
$$
\E\left(\int_{\tau}^{\tau'}e^{2\alpha s}|\eta^j_s|ds\right)\leq \frac{1}{8\xi\alpha}\E(e^{4\alpha \tau'}-e^{4\alpha \tau})+\frac{\xi}{2}\int_{0}^{T}\E|\eta^j_s|^2ds,
$$
where in the last inequality we have also used the fact that $0\leq \tau\leq \tau'\leq T.$
By using that $|e^{x}-e^{y}|\leq |x-y|\frac{1}{2}(e^x+e^y)$ in the previous inequality and the fact that $\tau'-\tau\leq \delta$, it follows that
$$
\E(\tau'-\tau)+\frac{1}{n^{3/2}}\sum_{j=1}\E\left(\int_{\tau}^{\tau'}e^{2\alpha s}|\eta^j_s|ds\right)\leq \delta+ \frac{\delta}{2\xi}e^{4\alpha T}+\frac{T\xi}{2n^{1/2}}\sup_{t\leq T}\max_{j\in [n]}\E(|\eta^j_s|^2).
$$
By using Proposition \ref{Prop:upper_bound_eta_it_square} and taking $\xi=\sqrt{\delta}$ in the inequality above, we deduce that
$$
\E\left(\left|\langle M^n(\varphi)\rangle_{\tau'}-\langle M^n(\varphi)\rangle_{\tau}\right|\right)\leq L(\delta+\sqrt{\delta}),
$$
showing that $(\langle M^n(\varphi)\rangle)_{n\geq 1}$ satisfies the condition \eqref{Def:aldous_condi_2} and thus the tightness of $(M^n(\varphi))_{n\geq 1}$.

{\bf Tightness of $(A^n(\varphi))_{n\geq 1}$.} Since $|A^{n}_t(\varphi)|\leq |M^{n}_t(\varphi)|$ for all $n\geq 1$ and $0\leq t\leq T$, and $(M^{n}(\varphi))_{n \geq 1}$ satisfies \eqref{Def:aldous_condi_1}, we have that $ (A^{n}(\varphi))_{n \geq 1}$ satisfies \eqref{Def:aldous_condi_1} as well. 

To show that tightness of $(A^n(\varphi))_{n\geq 1}$, it remains to show that $(A^n(\varphi))_{n\geq 1}$ satisfies the condition \eqref{Def:aldous_condi_2}. To that end, take stopping times $(\tau,\tau')\in St_{\delta}$, and note that
\begin{eqnarray*}
\P(|A^n_{\tau'}(\varphi)-A^n_{\tau}(\varphi)|>\epsilon_1)&\leq &\P(|M^n_{\tau'}(\varphi)(e^{-\alpha(\tau'-\tau)}-1)|>\epsilon_1/2)\\
&&+\P(|M^n_{\tau'}(\varphi)-M^n_{\tau}(\varphi)|>\epsilon_1/2)\\
&\leq &\frac{4}{\epsilon^2_1}\E\left((M^n_{\tau'}(\varphi))^2(e^{-\alpha(\tau'-\tau)}-1)^2\right)\\
&&+\P(|M^n_{\tau'}(\varphi)-M^n_{\tau}(\varphi)|>\epsilon_1/2)\\
&\leq & \frac{4(\alpha\delta)^2}{\epsilon^2_1}\E\left((M^n_{\tau'}(\varphi))^2\right)+\P(|M^n_{\tau'}(\varphi)-M^n_{\tau}(\varphi)|>\epsilon_1/2)\\
&\leq &\frac{4(\alpha\delta)^2}{\epsilon^2_1}\E\left((M^n_{T}(\varphi))^2\right)+\P(|M^n_{\tau'}(\varphi)-M^n_{\tau}(\varphi)|>\epsilon_1/2).
\end{eqnarray*}
Now, we can proceed as in the proof of the tightness of $(M^n(\varphi))_{n\geq 1}$ to conclude that $(A^n(\varphi))_{n\geq 1}$ satisfies the condition \eqref{Def:aldous_condi_2}, thus establishing the tightness of $(A^n(\varphi))_{n\geq 1}.$ 

{\bf Tightness of $(B^n(\varphi))_{n\geq 1}$.}  Fix $0\leq t\leq T$ and $\epsilon>0$. By using Markov inequality and then Jensen inequality, we have that for all $n\geq 1$ and $K>0$, (recall the definition of $B^n_t(\varphi)$),
$$
\P(|B^n_t(\varphi)|>K)\leq \frac{1}{K^2}\frac{1}{n}\sum_{i=1}^n\varphi^2(x_i)\E(\left (B^i_t)^2 \right).
$$
The inequality above and \eqref{Control_on_Bi_t} yield
$$
\P(|B^n_t(\varphi)|>K)\leq \frac{ L}{K^2}\sup_{n\geq 1} \sup_{0\leq s\leq T}\max_{j\in [n]} \E\left((\eta^j_s)^2\right).
$$
Since $ \sup_{n\geq 1} \sup_{0\leq s\leq T}\max_{j\in [n]}\E\left((\eta^j_s)^2\right)$ is finite by Proposition \ref{Prop:upper_bound_eta_it_square}, we deduce from inequality above that condition \eqref{Def:aldous_condi_1} holds.

We will now check that condition \eqref{Def:aldous_condi_2} holds as well.  
In the sequel, let $\Delta^j_s(f)=f(U^j_s)-f(u(s,x_j))$ for $1\leq j\leq n$, $0\leq s\leq T$ and $x_j\in [0,1].$
Take stopping times $(\tau,\tau')\in St_{\delta}$, and note that
\begin{multline}
\E |B^n_{\tau'}(\varphi)-B^n_{\tau}(\varphi)|\leq \E\left|\frac{1}{n}\sum_{i=1}^n \varphi(x_i)\frac{1}{n^{1/2}}\sum_{j=1}^n\int_{\tau}^{\tau'}e^{-\alpha(\tau'-s)}w(x_j,x_i)\Delta^j_s(f)ds\right| \\
+ \E\left|\frac{1}{n}\sum_{i=1}^n \varphi(x_i)\frac{1}{n^{1/2}}\sum_{j=1}^n\int_{0}^{\tau}(e^{-\alpha(\tau-s)}-e^{-\alpha(\tau'-s)})w(x_j,x_i)\Delta^j_s(f)ds\right|.
\end{multline}
Since $w$ and $\varphi$ are bounded functions, $f$ is Lipschitz continuous and $|e^{-x}-e^{-y}|\leq |x-y|$ for all $x,y\geq 0$, we have that
\begin{multline*}
  \E\left|\frac{1}{n}\sum_{i=1}^n \varphi(x_i)\frac{1}{n^{1/2}}\sum_{j=1}^n\int_{0}^{\tau}(e^{-\alpha(\tau-s)}-e^{-\alpha(\tau'-s)})w(x_j,x_i)\Delta^j_s(f)ds\right| \\
  \leq \alpha\|w\|_0\|\varphi\|_0L_{f}\frac{1}{n}\sum_{j=1}^n\E\int_{0}^{\tau}|\tau'-\tau||\eta^j_s|ds \leq \delta L \int_{0}^{T} \frac{1}{n} \sum_{j=1}^n\E|\eta^j_s|ds
\end{multline*}
so that Proposition \ref{Prop:upper_bound_eta_it_square} implies that
\begin{equation}
\label{Control_second_term}
\sup_{n\geq 1}\sup_{(\tau',\tau)\in St_{\delta}}\E\left|\frac{1}{n}\sum_{i=1}^n \varphi(x_i)\frac{1}{n^{1/2}}\sum_{j=1}^n\int_{0}^{\tau}(e^{-\alpha(\tau-s)}-e^{-\alpha(\tau'-s)})w(x_j,x_i)\Delta^j_s(f)ds\right|\leq L\delta.    
\end{equation}
Similarly, one can check that
\begin{multline}
\label{Control_first_term}
\E\left|\frac{1}{n}\sum_{i=1}^n \varphi(x_i)\frac{1}{n^{1/2}}\sum_{j=1}^n\int_{\tau}^{\tau'}e^{-\alpha(\tau'-s)}w(x_j,x_i)\Delta^j_s(f)ds\right|\\ \leq \|f'\|_{0}\|w\|_0\|\varphi\|_0\frac{1}{n}\sum_{j=1}^n\E\int_{\tau}^{\tau'}e^{-\alpha(\tau'-s)}|\eta^j_s|ds.    
\end{multline}
By applying Young's inequality we have that for all $\xi>0$, 
$$
e^{-\alpha(\tau'-s)}|\eta^j_s|\leq  \frac{1}{2\xi}e^{-2\alpha(\tau'-s)}+\frac{\xi}{2} |\eta^j_s|^2, 
$$
so that 
\begin{equation*}
\E\int_{\tau}^{\tau'}e^{-\alpha(\tau'-s)}|\eta^j_s|ds\leq \frac{1}{2\xi}\E\int_{\tau}^{\tau+\delta}e^{-2\alpha(\tau'-s)}ds+ \frac{\xi}{2} \int_{0}^{T}\E|\eta^j_s|^2ds \leq \frac{\delta}{2\xi}+ \frac{\xi}{2} \int_{0}^{T}\E|\eta^j_s|^2ds. 
\end{equation*}
As a consequence, by taking $\xi=\sqrt{\delta}$, it follows from \eqref{Control_first_term} and inequality above that
\begin{multline*}
\E\left|\frac{1}{n}\sum_{i=1}^n \varphi(x_i)\frac{1}{n^{1/2}}\sum_{j=1}^n\int_{\tau}^{\tau'}e^{-\alpha(\tau'-s)}w(x_j,x_i)\Delta^j_s(f)ds\right|\leq \sqrt{\delta}L(1+\int_{0}^{T}\sup_{n\geq 1}\frac{1}{n}\sum_{j=1}^n\E|\eta^j_s|^2ds).    
\end{multline*}
By Proposition \eqref{Prop:upper_bound_eta_it_square}, inequality above and \eqref{Control_second_term}, it follows then that
$$
\sup_{n\geq 1}\sup_{(\tau',\tau)\in St_{\delta}}\E |B^n_{\tau'}(\varphi)-B^n_{\tau}(\varphi)|\leq L(\delta +\sqrt{\delta}),$$
proving that \eqref{Def:aldous_condi_2} holds.

{\bf Proof of \eqref{control_of_C_phi}}. In the proof of Proposition \eqref{Prop:upper_bound_eta_it_square}, it has been proved that for any $0\leq t\leq T$ and $1\leq i\leq n$,
$$
|C^i_t|\leq \frac{1}{2n^{1/2}}\max_{j\in [n]}\sup_{0\leq s\leq t,y\in [0,1]}\left | \partial_1w(y,x_j)f(u(s,y)) +w(y,x)f'(u(s,y))\partial_2 u(s,y)\right |,
$$
so that 
$$
\sup_{0\leq t\leq T}|C^n_t(\varphi)|\leq \frac{\|\varphi\|_0}{2n^{1/2}}\max_{j\in [n]}\sup_{0\leq s\leq T,y\in [0,1]}\left | \partial_1w(y,x_j)f(u(s,y)) +w(y,x)f'(u(s,y))\partial_2 u(s,y)\right |,
$$
and the result follows.
\end{proof}

Since $(A^n(\varphi))_{n\geq 1}$ is càdlàg tight, $(B^n(\varphi))_{n\geq 1}$ is continuous tight and $(C^n(\varphi))$ goes to $0$, the following result holds (see \cite[Corollary VI.3.33.]{jacod1979calcul}). Furthermore, Corollary \ref{cor:tightness_of_Gamma_M_in_Sprime} is granted by \cite[Theorem 4.1]{Mitoma_89}.

\begin{corollary}\label{prop:tightness_of_Gamma:in:R}
Let us make the same assumptions as in Proposition \ref{prop:tightness_of_A_B_M:in:R}. Then, for each fixed $\varphi$ in $\cS$, the sequence of stochastic processes $(\Gamma^n(\varphi))_{n\geq 1}$ is tight in $\mathcal{D}([0,T],\bR)$.
\end{corollary}

\begin{corollary}\label{cor:tightness_of_Gamma_M_in_Sprime}
Let us make the same assumptions as in Proposition \ref{prop:tightness_of_A_B_M:in:R}. Then, the sequences of the laws of $(\Gamma^n)_{n\geq 1}$ and $(M^n)_{n\geq 1}$ are tight in $\mathcal{D}([0,T],\cS')$.
\end{corollary}

Furthermore, the limit trajectories of $(\Gamma^n)_{n\geq 1}$ and $(M^n)_{n\geq 1}$ are continuous as stated below.

\begin{proposition}\label{prop:continuity:limit:Gamma}
Suppose that $w$ is bounded. Then the limit points of $(\Gamma^n)_n$ and $(M^n)_{n\geq 1}$ are supported by $\mathcal{C}([0,T],\mathcal{S}')$.
\end{proposition}
\begin{proof}
We only give here the proof for the sequence $(\Gamma^n)_n$ but the same argument can be applied to $(M^n)_n$. For completeness, let us mention that the result for $(M^n)_n$ is also a byproduct of Proposition \ref{prop:convergence:martingale}.

By \cite[Theorem 13.4.]{Billing_Convergence} it suffices to show that the maximum jump size of $\Gamma^n$ goes to $0$ in probability. This means that we need to prove that, for any bounded set $B\subset \cS$, the random variable $J_B(\Gamma^n)=\sup_{0\leq t\leq T }\sup_{\varphi\in B}|\Gamma^n_t(\varphi)-\Gamma^n_{t-}(\varphi)|$ converges in probability to $0$ as $n\to\infty.$ 
Recall that $B\subset \cS$ is bounded if $\sup_{\varphi\in B}\|\varphi\|_{k}<\infty$ for any $k\geq 0$ (which is different from boundedness with respect to $d_{\cS}$, see \cite[Definition 2.9]{simon2017banach}). In particular, we have $\sup_{\varphi\in B}\|\varphi\|_0<\infty.$ 
Now, observe that (see also proof of condition 2 in the proof of Proposition \ref{prop:convergence:martingale})
$$
\sup_{0\leq t\leq T}\sup_{\varphi\in B}|M^{n}_t(\varphi)-M^{n}_{t-}(\varphi)|\leq e^{\alpha T} \frac{\|w\|_0}{\sqrt{n}}\sup_{\varphi\in B}\|\varphi\|_0\sup_{t\leq T}\sum_{j=1}^n\Delta N^j_s,
$$
where $\Delta N^j_s=N^j_s-N^j_{s-}$ for each $j\in [n]$ and $s\geq 0$.
Almost surely for all $1\leq i,j\leq n$ with  $i\neq j$, the counting processes $N^j$ and $N^i$ never jump simultaneously, so that  
$$
\sup_{0\leq t\leq T}\sum_{j=1}^n\Delta N^j_s\leq 1 \ \mbox{almost surely,}
$$
and therefore almost surely
$$
\sup_{0\leq t\leq T}\sup_{\varphi\in B}|M^{n}_t(\varphi)-M^{n}_{t-}(\varphi)|\leq e^{\alpha T} \frac{\|w\|_0}{\sqrt{n}}\sup_{\varphi\in B}\|\varphi\|_0.
$$
Finally, since $\Gamma^n_t(\varphi)=e^{-t\alpha}M^n_t(\varphi)+B^n_t(\varphi)+C^n_t(\varphi)$ and both $B^n_t(\varphi)$ and $C^n_t(\varphi)$ are continuous functions of time, we deduce that $|\Gamma^n_t(\varphi)-\Gamma^n_{t-}(\varphi)|\leq |M^{n}_t(\varphi)-M^{n}_{t-}(\varphi)|$. Thus, it follows that almost surely,
$$
J_B(\Gamma^n)\leq e^{\alpha T} \frac{\|w\|_0}{\sqrt{n}}\sup_{\varphi\in B}\|\varphi\|_0,
$$
implying the result.
\end{proof}

Since the limit trajectories of $(M^n)_{n\geq 1}$ and $(\Gamma^n)_{n\geq 1}$ are continuous, we have the joint tightness \cite[Corollary VI.3.33]{Jacod_2003}.

\begin{corollary}\label{cor:joint:tightness}
Under assumptions of Proposition \ref{prop:tightness_of_A_B_M:in:R}, the sequence of the laws of $(\Gamma^n,M^n)_{n\geq 1}$ is tight in the space $\mathcal{D}([0,T],\mathcal{S}'\times \mathcal{S}')$ with limit points in $\mathcal{C}([0,T],\mathcal{S}'\times \mathcal{S}')$.
\end{corollary}

\section{Limit equation}
\label{sec:limit_equation}

In this section we first show the convergence of the local martingale $(M^n)_{n\geq 1}$ in order to state the limit equation \eqref{eq:limit:equation:Gamma} satisfied by the limit points of $(\Gamma^n)_{n\geq 1}$.

\begin{definition}
\label{def:gaussian:process}
Let $M$ be a continuous centered Gaussian process with values in $\mathcal{S}'$ with covariance given, for all $\varphi_1$ and $\varphi_2$ in $\mathcal{S}$, for all $t_1$ and $t_2\geq 0$, by
\begin{equation}
    \E\left[M_{t_1}(\varphi_1)M_{t_2}(\varphi_2) \right] = \int_0^{t_1\wedge t_2} \int_{0}^1 e^{2\alpha s}  I[\varphi_1](y) I[\varphi_2](y) f(u(s,y)) dy ds,
\end{equation}
where for each $y\in [0,1]$,
\begin{equation*}
    I[\varphi](y)= \int_0^1 w(y,x)\varphi(x) dx.
\end{equation*}
\end{definition}

\begin{proposition}\label{prop:convergence:martingale}
Under assumptions of Proposition \ref{prop:tightness_of_A_B_M:in:R}, the sequence $(M^n)_{n\geq 1}$ of processes in
$\mathcal{D}(\bR_+,\mathcal{S}')$ converges in law to $M$ defined above.
\end{proposition}
\begin{proof}
By Corollary \ref{cor:tightness_of_Gamma_M_in_Sprime} and \cite[Theorem 5.3]{Mitoma_89}, it suffices to show that for $0\leq t_1\leq t_2\leq \ldots\leq t_m\leq T$ and $\varphi_1,\ldots,\varphi_m\in \cS$, the sequence of random vectors $(M^n_{t_1}(\varphi_1),\ldots, M^n_{t_m}(\varphi_m))_{n\geq 1}$ converges in law to a Gaussian random vector $\cN(0,\Sigma)$ with covariance matrix $\Sigma=(\Sigma_{ij})_{1\leq i,j\leq m}$ given by
$$
(\Sigma)_{ij}=\int_0^{t_i\wedge t_j} \int_{0}^1 e^{2\alpha s} I[\varphi_i](y) I[\varphi_j](y) f(u(s,y)) dy ds.
$$
To show this convergence holds, by Cramér--Wold Theorem, it is enough to show that for all $\beta=(\beta_1,\ldots,\beta_m)$, the sequence of random variables $(\sum_{p=1}^m\beta_pM^n_{t_p}(\varphi_p))_{n\geq 1}$ converges in law to a Gaussian random variable $\cN(0,\sigma^2)$ with variance $\sigma^2=\beta^{T}\Sigma\beta.$
To that end, we will resort to the central limit theorem for local martingales from \cite{Rebolledo_80}.

In what follows, for each $n\geq 1$, let $(\tilde{M}^n_t)_{t\geq 0}$ be the local martingales defined by
$$
\tilde{M}^n_t=\sum_{p=1}^m\beta_pM^n_{t_p\wedge t}(\varphi_p).
$$
Its angle bracket can be written in the following form:
$$
\langle \tilde{M}^n \rangle_t=\sum_{1\leq p,q\leq m}\beta_p\beta_q \langle M^n(\varphi_p), M^n(\varphi_q) \rangle_{(t_p\wedge t_q) \wedge t}. 
$$
 According to \cite{Rebolledo_80}, the sequence $(\tilde{M}^n)_{n\geq 1}$ converges in law to a centered Gaussian process $\tilde{M}$ with covariance function $\mCov(\tilde{M}_t,\tilde{M}_s)=C(t\wedge s)$, if the following conditions are verified:
 \begin{enumerate}
     \item $\langle \tilde{M}^n \rangle_t$ converges to $C(t)$ in probability as $n\to\infty$, for each $t>0.$
     \item For each $\epsilon>0$ and $t>0$, the sequence of random variables
     $$
     \sum_{s\leq t}|\Delta \tilde{M}^n_s|1_{\{|\Delta \tilde{M}^n_s|>\epsilon\}},
     $$
     converges to $0$ in probability as $n\to\infty.$
 \end{enumerate}
Let us assume that these two conditions have been checked with $(C(t))_{t\geq 0}$ such that $C(T)=\sigma^2$. In that case,  we would have that $\sum_{p=1}^m \beta_p M^n_{t_p}(\varphi_p)=\tilde{M}^n_T$ converges to $\cN(0,\sigma^2)$ in law as $n\to\infty$, concluding the proof of the proposition. 
In the remaining part of the proof, we will check conditions 1 and 2 above are satisfied with $(C(t))_{t\geq 0}$ such that $C(T)=\sigma^2.$

{\bf Proof of condition 1.} For each $t\geq 0$, write 
$$
C(t)=\sum_{1\leq p,q\leq m}\beta_p\beta_q\int_{0}^{(t_q\wedge t_p)\wedge t}I[\varphi_p](y)I[\varphi_q](y)e^{2\alpha s}f(u(s,y))dyds,
$$
and for each $n\geq 1$, $y\in [0,1]$ and $\varphi\in \cS$, let
$$
I^n[\varphi](y)=\frac{1}{n}\sum_{i=1}^n\varphi(x_i)w(y,x_i).
$$
Observe that $C(T)=\sigma^2$ and $\|I^n[\varphi]\|_0\leq \|\varphi\|_0\|w\|_0$. Now, from the Riemann sum approximation, we also have that
$$
\|I^n[\varphi]-I[\varphi]\|_0\leq \frac{1}{2n}\sup_{x,y\in [0,1]}|\varphi'(x)w(y,x)+\varphi(x)\partial_2w(y,x)|.
$$
Moreover, with this notation, it follows from \eqref{def:_angle_bibracket_Martinga_M_n_phi_t} that
$$
\langle \tilde{M}^n \rangle_t=\sum_{1\leq p,q\leq m}\beta_q\beta_p\frac{1}{n}\sum_{j=1}^n\int_{0}^{t\wedge(t_p\wedge t_q)}I^n[\varphi_p](x_j)I^n[\varphi_q](x_j)e^{2\alpha s}f(U^j_s)ds, 
$$
so that 
\begin{multline}
\label{difference_angle_Mtilde_C_t}
\E|\langle \tilde{M}^n \rangle_t-C(t)|\leq \sum_{1\leq p,q\leq m}\beta_q\beta_p\E\left|\frac{1}{n}\sum_{j=1}^n\int_{0}^{t\wedge(t_p\wedge t_q)} I^n[\varphi_p](x_j)I^n[\varphi_q](x_j)e^{2\alpha s}f(U^j_s)ds\right.\\\left.
-\int_{0}^{t\wedge(t_p\wedge t_q)}\int_{0}^1I[\varphi_p](y)I[\varphi_q](y)e^{2\alpha s}f(u(s,y))dyds\right|.
    \end{multline}
Now, for any $1\leq p,q\leq m$, one can check that
\begin{multline}
\label{proof_cond_1:ineq_1}
\left|\frac{1}{n}\sum_{j=1}^n(I^n[\varphi_p](x_j)I^n[\varphi_q](x_j)-I[\varphi_p](x_j)I[\varphi_q](x_j))\right|\\ \leq \frac{\|w\|_0}{n}\max_{u,v\in\{p,q\}}\left\{\|\varphi_v\|_0\sup_{x,y\in [0,1]}|\frac{\partial}{\partial x}(\varphi_u(x)w(y,x))|\right\}.    
\end{multline}
Similarly, from Riemann sum approximation, we have for any $s\leq t$ fixed, 
\begin{multline}
\label{proof_cond_1:ineq_2}
\left|\frac{1}{n}\sum_{j=1}^nI[\varphi_p](x_j)I[\varphi_q](x_j)f(u(s,x_j))-\int_{0}^1I[\varphi_p](y)I[\varphi_q](y)f(u(s,y))dy\right|\\ \leq \frac{1}{2n}\sup_{y\in [0,1],h\leq t}\left|\frac{d}{dy}(I[\varphi_p](y)I[\varphi_q](y)f(u(h,y)))\right|.    
\end{multline}
The local boundedness of both $f$ and $u$ implies $\sup_{t\leq T}\|f(u(h,\cdot)\|_0<\infty$. Combining this fact with
the boundedness of $f^(1)$, and the smoothness of both $\varphi_p$ and $\varphi_q$, one can show that
$$
\sup_{y\in [0,1],h\leq T}\left|\frac{d}{dy}(I[\varphi_p](y)I[\varphi_q](y)f(u(h,y)))\right|<\infty$$
Furthermore, we have that
\begin{multline}
\label{proof_cond_1:ineq_3}
\E\left|\frac{1}{n}\sum_{j=1}^n\int_{0}^{t\wedge(t_p\wedge t_q)}e^{2\alpha 2}I[\varphi_p](x_j)I[\varphi_q](x_j)(f(U^j_s)-f(u(s,x_j)))ds \right|\\ \leq \|f'\|_{0} \frac{\|I[\varphi_p]\|_0\|I[\varphi_q]\|_0}{n^{1/2}}\int_{0}^{t} e^{2\alpha s}\sup_{n\geq 1}\frac{1}{n}\sum_{j=1}^n\E|\eta^j_s|ds. \end{multline}
Combining the inequalities \eqref{proof_cond_1:ineq_1}, \eqref{proof_cond_1:ineq_2} and \eqref{proof_cond_1:ineq_3} with \eqref {difference_angle_Mtilde_C_t}, and using that the function $s\mapsto \sup_{n\geq 1}\frac{1}{n}\sum_{j=1}^n\E|\eta^j_s|$ is locally bounded,  we then have that there exists a constant $C$ not depending on $n$ such that for all $n$ sufficiently large
$$
\E|\langle \tilde{M}^n \rangle_t-C(t)|\leq C\frac{1}{n^{1/2}}.
$$
The proof of condition 1 follows now from Markov inequality.

{\bf Proof of condition 2.} It is enough to prove that there exists a positive constant $C$ not depending on $n\geq 1$ such that for each $t>0$ and $n\geq 1$, 
$$
\sup_{0\leq s\leq t}|\Delta\tilde{M}^n_s|\leq \frac{C}{n^{1/2}} \ \mbox{almost surely.}
$$
To prove that, observe that for each $s\geq 0$, 
$$
|\Delta\tilde{M}^n_s|\leq m \max_{1\leq p\leq m}\{\|\varphi_p\|_0\}\|w\|_0\frac{1}{n^{1/2}}\sum_{j=1}^n\Delta N^j_s,
$$
implying that (with $C=m \max_{1\leq p\leq m}\{\|\varphi_p\|_0\}\|w\|_0$)
$$
\sup_{0\leq s\leq t}|\Delta\tilde{M}^n_s|\leq \frac{C}{n^{1/2}}\sup_{0\leq s\leq t}\sum_{j=1}^n\Delta N^j_s.
$$
Since almost surely for all $1\leq i,j\leq n$ with  $i\neq j$, the counting processes $N^j$ and $N^i$ never jump simultaneously, it follows that  
$$
\sup_{0\leq s\leq t}\sum_{j=1}^n\Delta N^j_s\leq 1 \ \mbox{almost surely,}
$$
and the result follows.
\end{proof}

We are now in position to state the limit equation satisfied by the limit points, generically denoted by $\Gamma$, of the sequence $(\Gamma^n)_n$.
The limit equation is:

\begin{equation}
\label{eq:limit:equation:Gamma}
    \forall \varphi \in \mathcal{S}, \quad 
\Gamma_t(\varphi)=e^{-\alpha t}M_t(\varphi)+\int_{0}^{t}e^{-\alpha(t-s)}\Gamma_s\left(\int_{0}^1\varphi(x)w(\cdot,x)f'(u(s,\cdot))dx\right)ds,
\end{equation}
where $M$ is the Gaussian process of Definition \ref{def:gaussian:process}.

For notational simplicity, let us introduce the following maps $F_{\varphi}:\mathcal{D}([0,T],\mathcal{S}'\times \mathcal{S}')\to \mathcal{D}([0,T],\bR)$ defined, for all $(g_t,m_t)_{t\in [0,T]}$ in $\mathcal{D}([0,T],\mathcal{S}'\times\mathcal{S}')$, by
\begin{equation*}
    F_{\varphi}(g,m)_t = g_t(\varphi) - e^{-\alpha t} m_t(\varphi) - \int_{0}^{t} e^{-\alpha(t-s)} g_s \left(\int_{0}^1\varphi(x)w(\cdot,x)f'(u(s,\cdot))dx\right)ds,
\end{equation*}
so that \eqref{eq:limit:equation:Gamma} is equivalent to : for all $\varphi\in \cS$, $F_{\varphi}(\Gamma,M)=0$.

\begin{remark}
Assumption \ref{ass:f} ensures that the function $y\mapsto \int_0^1\varphi(x)w(y,x)f'(u(s,y))dx$ is in $S$ so that the RHS of \eqref{eq:limit:equation:Gamma} is well defined.

About the space $\mathcal{D}([0,T],\mathcal{S}'\times \mathcal{S}')$, the product $\mathcal{S}'\times \mathcal{S}'$ is endowed with the product topology of the strong topology on $\mathcal{S}'$. Then, $\mathcal{D}([0,T],\mathcal{S}'\times \mathcal{S}')$ is endowed with some projective topology inspired by Skorokhod topology (see \ref{app:frechet} for more insight). Finally, the maps $F_{\varphi}$ are continuous with respect to that topology.
\end{remark}

\begin{proposition}
\label{prop:limit:law:limit:equation}
Under Assumption \ref{ass:f}, for all $\varphi\in \cS$, $F_{\varphi}(\Gamma^n,M^n)\to 0$ in probability.
\end{proposition}
\begin{proof}
Observe that for each $t\geq 0,$
$$
F_{\varphi}(\Gamma^n,M^n)_t=B^n_t(\varphi)+C^n_t(\varphi)-\frac{1}{n}\sum_{j=1}^n\int_{0}^1\varphi(x)w(x_j,x)dx\int_{0}^t e^{-\alpha(t-s)}\eta^j_sf'(u(s,x_j))ds.
$$
Recall that $\sup_{t\leq T}|C^n_t(\varphi)|\leq C\|\varphi\|_{\infty}n^{-1/2}$ by Proposition \ref{prop:tightness_of_A_B_M:in:R}. Thus given $\epsilon>0$, we have that 
$\sup_{t\leq T}|C^n_t(\varphi)|\leq \epsilon/2$ for all $n\geq (2C\|\varphi\|_0\epsilon^{-1})^2,$
so that the event $\left\{\sup_{t\leq T}|F_{\varphi}(\Gamma^n,M^n)_t|>\epsilon\right\}$ is contained in event
$$
\left\{\sup_{t\leq T}\left|B^n_t(\varphi)-\frac{1}{n}\sum_{j=1}^n\int_{0}^1\varphi(x)w(x_j,x)dx \int_{0}^t e^{-\alpha(t-s)}\eta^j_sf'(u(s,x_j))ds\right|>\epsilon/2\right\}.
$$
We will show in the remaining part of the proof that
$$
\sup_{0\leq t\leq T}\left|B^n_t(\varphi)-\frac{1}{n}\sum_{j=1}^n\int_{0}^1\varphi(x)w(x_j,x)dx \int_{0}^t e^{-\alpha(t-s)}\eta^j_sf'(u(s,x_j))ds\right|\to 0
$$
in $L^1$ as $n\to\infty$, implying the result.

To that end, note that for all $t\leq T$ and $n\geq 1,$
$$
\left|B^n_t(\varphi)-\frac{1}{n}\sum_{j=1}^n\int_{0}^1\varphi(x)w(x_j,x)dx \int_{0}^t e^{-\alpha(t-s)}\eta^j_sf'(u(s,x_j))ds\right|\leq 
I^n+II^n,
$$
where (remember that $\Delta^j_s(f)=f(U^j_s)-f(u(s,x_j))$)
$$
I^n=\left|\frac{1}{\sqrt{n}}\sum_{j=1}^n\int_0^t e^{-\alpha(t-s)}\left[\frac{1}{n}\sum_{i=1}^n\varphi(x_i)w(x_j,x_i)-\int_{0}^1\varphi(x)w(x_j,x)dx\right]\Delta_s^j(f)ds \right|,  
$$
and
\begin{equation}
\label{eq:linearization:II}
II^n=\left| \frac{1}{\sqrt{n}}\sum_{j=1}^n\int_{0}^1\varphi(x)w(x_j,x)dx\int_0^te^{-\alpha(t-s)}\left[\Delta_s^j(f) -f'(u(s,x_j))\frac{\eta^j_s}{\sqrt{n}}\right]\right|.
\end{equation}
Now, using the fact that $f$ is Lipschitz and classical Riemann estimates, it follows that
$$
I^n\leq \frac{\|f'\|_{0}}{2n}\sup_{0\leq x,y\leq 1}|\frac{\partial}{\partial x}(\varphi(x)w(y,x))|\int_{0}^{t}e^{-\alpha(t-s)}\frac{1}{n}\sum_{j=1}^{n}|\eta^j_s|ds.
$$
Since the function $s\to \sup_{n\geq 1}\frac{1}{n}\sum_{j=1}^{n}\E|\eta^j_s|$ is locally bounded by Proposition \ref{Prop:upper_bound_eta_it_square}), we have that $I^n\to 0$ in $L^1$ as $n\to\infty$. 

To deal with $II^n$, we recall that the Taylor approximation of order 2 yields
$$
|f(x)-f(y)-f'(y)(x-y)|\leq \frac{|x-y|^2}{2}\|f''\|_{0}, \ \mbox{for all} \ x,y\in\bR.
$$
Hence, we have for all $s\geq 0$ and $j\in[n],$
$$
|\Delta^j_s(f)-f'(u(s,x_j))\frac{\eta^j_s}{\sqrt{n}}|\leq \frac{(U^j_s-u(s,x_j))^2}{2}\|f''\|_0=\frac{\|f''\|_0}{2n}(\eta^j_s)^2,
$$
so that
$$
\E(II^n)\leq \sup_{0\leq y\leq 1}\left|\int_{0}^1\varphi(x)w(y,x)dx\right|\frac{\|f''\|_0}{2\sqrt{n}}\int_{0}^te^{-\alpha(t-s)}\sup_{m\geq 1}\frac{1}{m}\sum_{j=1}^m\E((\eta_s^j)^2)ds.
$$
The local boundedness of  $s\to \sup_{n\geq 1}\frac{1}{n}\sum_{j=1}^{n}\E((\eta^j_s)^2)$ implies that $II^n\to$ in $L^1$ as $n\to\infty$ as well, concluding the proof of the proposition.
\end{proof}

We are now in position to state the main result of this section.

\begin{thm}
\label{thm:limit:is:solution:of:limit:equation}
Under Assumption \ref{ass:f}, any limit point $\Gamma$ of the sequence $(\Gamma^n)_{n\geq 1}$ is a solution of \eqref{eq:limit:equation:Gamma} in $\mathcal{C}(\bR_+,\mathcal{S}')$.
\end{thm}
\begin{proof}
Let $\Gamma$ be a limit point of $(\Gamma^n)_n$ and $(n_k)_k$ be such that $\Gamma^{n_k}\to \Gamma$ in distribution. Like in Corollary \ref{cor:joint:tightness}, we obviously have joint tightness of $(\Gamma^{n_k},M^{n_k})$. Hence let $M$ be such that $(\Gamma,M)$ is a limit point of $(\Gamma^{n_k},M^{n_k})$. The convergence result of Proposition \ref{prop:limit:law:limit:equation} and continuous mapping theorem imply that for all $\varphi$ in $\cS$, $F_{\varphi}(\Gamma,M)=0$. Hence $\Gamma$ satisfies \eqref{eq:limit:equation:Gamma}. 

Finally, the continuity of $\Gamma$ follows from Proposition \ref{prop:continuity:limit:Gamma}.
\end{proof}

\section{Convergence}
\label{sec:convergence}

\begin{proposition}
\label{prop:pathwise:uniqueness}
Under Assumption \ref{ass:f}, there is path-wise uniqueness of the solutions of limit equation \eqref{eq:limit:equation:Gamma}: if $\Gamma$ and $\tilde{\Gamma}$ are two solutions in $\mathcal{C}(\bR_+, \mathcal{S}')$ constructed on the same probability space as $M$, then $\Gamma$ and $\tilde{\Gamma}$ are indistinguishable.
\end{proposition}
\begin{proof}
Let $\Gamma$ and $\tilde{\Gamma}$ be two solutions and take $T>0$. In the following, consider the restrictions of $\Gamma$ and $\tilde{\Gamma}$ to $[0,T]$. For almost every $\omega\in \Omega$, $\Gamma(\omega)$ and $\tilde{\Gamma}(\omega)$ are continuous and $F_{\varphi}(\Gamma(\omega)-\tilde{\Gamma}(\omega),M)=0$ for all $\varphi\in \cS$, i.e.
\begin{equation}\label{eq:equation:for:uniqueness:Gamma}
    (\Gamma(\omega)-\tilde{\Gamma}(\omega))_t(\varphi) = \int_0^t e^{-\alpha (t-s)} (\Gamma(\omega)-\tilde{\Gamma}(\omega))_s \left( \int_0^1\varphi(x)w(\cdot,x)f'(u(s,\cdot))dx \right) ds.
\end{equation}
In the following, we assume that such a generic $\omega$ is fixed and omit to write the dependence.

Let $\varphi$ be in $\cS$. Since $\Gamma$ and $\tilde{\Gamma}$ are continuous, $\Gamma(\varphi)$ and $\tilde{\Gamma}(\varphi)$ belong to $\cC([0,T],\bR)$ and in particular, $\sup_{t\in [0,T]} |(\Gamma_t-\tilde{\Gamma}_t)(\varphi)| <+\infty$. The uniform boundedness principle \cite[Theorem 10.11.]{simon2017banach} implies that there exists $k$ and $c>0$ such that
\begin{equation*}
    \forall \varphi\in \cS, \, \sup_{s\in [0,T]} |(\Gamma_s-\tilde{\Gamma}_s)(\varphi)| \leq c ||\varphi||_k.
\end{equation*}
Therefore, let us define for all $t$ in $[0,T]$, 
\begin{equation*}
    L_t = \sup_{\varphi \in \cS} \sup_{s\in [0,t]} \frac{|(\Gamma_s-\tilde{\Gamma}_s)(\varphi)|}{||\varphi||_k} \leq c <+\infty.
\end{equation*}
In order to use boundedness for $f$ and its derivatives, let us remark that $K_T = \sup_{t\in [0,T]} ||u(t,\cdot)||_0<+\infty$ so that without loss of generality, $f$ could be restricted to the compact interval $[-K_T,K_T]$ and so considered to be $\cC^\infty$ with bounded derivatives of any order.

Then, using Lemma \ref{lem:Ck:bound} and the fact that $w$, $u$ and $f$ are smooth,
\begin{multline*}
    \left\| \int_0^1\varphi(x)w(\cdot,x)f'(u(s,\cdot))dx \right\|_k \leq  ||\varphi||_0 \sup_{s\in [0,T]} \sup_{x\in [0,1]} ||w(\cdot,x)f'(u(s,\cdot))||_k \\
    \leq C ||\varphi||_k  \left( \sup_{x\in [0,1]} ||w(\cdot,x)||_k \right) ||f_{|[-K_T,K_T]}||_{k+1} \left( \sup_{s\in [0,T]} ||u(s,\cdot)||_{k}^k \right) \leq C_T ||\varphi||_k.
\end{multline*}
Going back to \eqref{eq:equation:for:uniqueness:Gamma}, we have
\begin{equation*}
    |(\Gamma-\tilde{\Gamma})_t(\varphi)| \leq  \int_0^t  L_s \left\| \int_0^1\varphi(x)w(\cdot,x)f'(u(s,\cdot))dx \right\|_k ds.
\end{equation*}
and so $L_t \leq C_T \int_0^t  L_s  ds$. Since $L_t$ is a priori bounded by $c$, Gronwall's lemma implies that for all $t\in [0,T]$, $L_t=0$ which means that $\Gamma=\tilde{\Gamma}$.

The argument above holds for almost every $\omega$ so path-wise uniqueness is proven.
\end{proof}

\begin{thm}
\label{thm:CLT}
Under Assumption \ref{ass:f}, the sequence $(\Gamma_n)_n$ converges in law in $\mathcal{D}(\bR_+, \mathcal{S}')$ to the unique solution of \eqref{eq:limit:equation:Gamma} in $\mathcal{C}(\bR_+, \mathcal{S}')$.
\end{thm}

\begin{proof}
Let $\Gamma$ be a limit point of $(\Gamma_n)_n$. According to Theorem \ref{thm:limit:is:solution:of:limit:equation}, $\Gamma$ is a solution of the limit system \eqref{eq:limit:equation:Gamma} in
$\mathcal{C}(\bR_+,\mathcal{S}')$. Yet, path-wise uniqueness given by Proposition \ref{prop:pathwise:uniqueness} and Yamada--Watanabe theorem gives weak uniqueness by the same argument as in \cite[Theorem IX.1.7(i)]{Revuz_1999}. Finally, weak uniqueness gives uniqueness of the limit point $\Gamma$ and so convergence.
\end{proof}

\section{Connection with a stochastic NFE}
\label{sec:conec_stoc_NFE}

Let us begin this section with some discussion about the standard central limit theorem. Let $\bar{X}_n$ be the empirical mean of some i.i.d. square integrable centered and normalized random variables $X_1,\dots,X_n$. The law of large numbers and central limit theorem respectively tells that $\bar{X}_n = 0 + o(1)$ and $\bar{X}_n = 0 + n^{-1/2} Z + o(n^{-1/2})$ where $Z$ is a standard Gaussian random variable. Of course, the second statement is purely informal but gives the flavor of the result.

With this description in mind, we provide here an informal overview of the mean-field limit and central limit theorem stated in this paper. Assume for ease of the presentation that the limit $\Gamma_t$ is in fact a function, namely there exists $(t,x)\mapsto \Gamma_t(x)$ such that for all $\varphi$, $\Gamma_t(\varphi) = \int_0^1 \varphi(x) \Gamma_t(x)dx$. Then, the take-away message until this point of the paper is: in order to approximate the microscopic system $(U^i_t)_{t\geq 0,i \in [n]}$, there are two steps,
\begin{enumerate}
    \item the mean-field limit $u(t,x)$ makes an error of order $n^{-1/2}$ since the renormalized error $\Gamma^n$ goes to something non trivial;
    \item the mean-field combined with the fluctuations, namely $Y^n(t,x):=u(t,x)+n^{-1/2}\Gamma_t(x)$ makes an error of order $o(n^{-1/2})$.
\end{enumerate}
With this in mind, the goal of this section is to find an approximation $V^n$ of $Y^n$ with an error of order $o(n^{-1/2})$ justifying that $V^n$ is a better approximation than the standard mean-field limit. This approximation $V^n$ will be characterized as the unique solution of a particular stochastic version of the neural field equation and the discussion above justifies that it is the \emph{Stochastic Neural Field Equation} (SNFE) naturally associated with the Hawkes processes given by \eqref{def:SEHP_intensity}.

In the following we are interested in the following SNFE
\begin{equation}
\label{def:SNFE}
\begin{cases}
dV^{n}_t(x)=\left( -\alpha V^{n}_t(x)+ \int_{0}^1 w(y,x) f(V^{n}_t(y))dy \right) dt + \int_{0}^1 w(y,x) \frac{\sqrt{f(V^{n}_t(y))}}{\sqrt{n}} W(dt,dy),\\
V^{n}_0(x)=u_0(x).
\end{cases}
\end{equation}
where $W$ is a Gaussian white noise. The mathematical arguments used below are highly inspired from \cite{faugeras2015stochastic} where other kinds of stochastic neural field equations can be found.

First we need to specify what we mean by Gaussian white noise. Here, we use the Gaussian random field 
\begin{equation}\label{def:WNP}
W=(W(A))_{A\in\cB(\bR_+\times [0,1])},
\end{equation}
with covariance function
\begin{equation*}
\E(W(A)W(B))=|A\cap B|,    
\end{equation*}
where $|A\cap B|$ denotes the Lebesgue measure of $A\cap B$. Then, the SNFE \eqref{def:SNFE} has to be understood in the weak sense.
\begin{definition}\label{def:solution:snfe}
By a solution to \eqref{def:SNFE} we mean a real-valued random field $(V^{n}_t(x))_{t\geq 0,x\in [0,1]}$ such that
\begin{multline}
\label{def:solution_SNFE}
V^{n}_t(x) = e^{-t} V^{n}_0(x) + \int_{0}^{t} e^{-(t-s)}\int_{0}^1 w(y,x) f(V^{n}_s(y))dyds 
\\ +\int_{0}^{t} e^{-(t-s)} \int_{0}^1 w(y,x) \frac{\sqrt{f(V^{n}_s(y))}}{\sqrt{n}} W(ds,dy),
\end{multline}
almost surely for all $t\geq 0$ and $x\in [0,1].$
\end{definition}
Now it suffices to give sense to the stochastic term,
$$
\int_{0}^{t} e^{-(t-s)} \int_{0}^1 w(y,x) \frac{\sqrt{f(V^{n}_s(y))}}{\sqrt{n}} W(ds,dy).
$$
Walsh's theory of stochastic integration provides a nice framework to give it a sense, see for instance \cite[Theorem 3.1]{faugeras2015stochastic} or \cite{Dalang2009AMO} for details. We use this theory in the rest of the paper.

Before stating the well posedness of the SFNE and the approximation result, we first provide some heuristics leading to \eqref{def:SNFE}.

\subsection{Heuristic motivation for \eqref{def:SNFE}}
\label{sec:heuristics}
Assume for now that $\Gamma_t$ and $M_t$ defined as distributions in the previous sections are in fact functions (this will be precised later on). Then, let us make the following abuse of notation: $\Gamma_t(x) = \Gamma_t(\delta_x)$ and $M_t(x) = M_t(\delta_x)$ in such a way that,
for any $\varphi\in \cS$, 
$$
\Gamma_t(\varphi)=\int_0^1\varphi(x)\Gamma_t(x)dx \ \mbox{and} \ M_t(\varphi)=\int_0^1\varphi(x)M_t(x)dx.
$$
To guess what equation $\Gamma_t(x)$ solves, we take (informally) $\varphi=\delta_x$ in \eqref{def:limit_process} to get
\begin{equation}
\label{def:limit_process_as_function}
\Gamma_t(x)=e^{-\alpha t}M_t(x)+\int_{0}^t e^{-\alpha (t-s)}\int_{0}^1w(y,x)f'(u(s,y))\Gamma_s(y)dyds, \ t\geq 0 \ \mbox{and} \ x\in[0,1].
\end{equation}
As briefly discussed in the Introduction, the first term on the RHS of the equation above, namely $M_t(x)$, is the limit in distribution as $n\to\infty$ of
$$
\frac{1}{\sqrt{n}}\sum_{j=1}^n w(x_j,x)\int_{0}^t e^{\alpha s}(dN^j_s-f(U^j_s))ds,
$$
which has mean zero and limit variance
$$
\int_{0}^te^{2\alpha s}\int_0^1w^2(y,x)f(u(s,y))dyds.
$$
Besides, the Martingale Central Limit theorem ensures that $M_t(x)$ is Gaussian so a suitable description for $M_t(x)$ should be:
\begin{equation}
\label{def:martingale_as_function}
M_t(x)=\int_{0}^t\int_{0}^1e^{\alpha s}w(y,x)\sqrt{f(u(s,y))}W(ds,dy), \ t\geq 0 \ \mbox{and} \ x\in[0,1],   
\end{equation}
where $W$ is the white noise process defined in \eqref{def:WNP}.

Now, we are interested in $Y^n_t(x)=u(t,x)+n^{-1/2}\Gamma_t(x)$ which, by summing \eqref{eq:system:u:lambda} and \eqref{def:limit_process_as_function} (where $M$ is replaced according to \eqref{def:martingale_as_function}), is given by
\begin{multline}\label{eq:heuristics:Yn}
 Y^n_t(x) = e^{-\alpha t}u_0(x_i) +\int_{0}^{t}e^{-\alpha(t-s)}\int_0^1w(y,x)\frac{\sqrt{f(u(s,y))}}{\sqrt{n}}W(ds,dy)\\
  + \int_{0}^{t}e^{-\alpha(t-s)}\int_0^1w(y,x)\left(f(u(s,y))+n^{-1/2}f'(u(s,y))\Gamma_s(y)\right)dyds.
\end{multline}
According to Taylor approximation, the error made when replacing $f(Y^n_s(y))$ by $f(u(s,y))+n^{-1/2}f'(u(s,y))\Gamma_s(y)$ is of order $o(n^{-1/2})$ and replacing $f(Y^n_s(y))$ by $f(u(s,y))$ is of order $o(1)$. When making these replacements, Equation \eqref{eq:heuristics:Yn} is exactly the equation satisfied by the solution of our SNFE. For this reason, we expect that the difference between $Y^n$ and the solution of the SFNE is of order $o(n^{-1/2})$. This result is confirmed below in Theorem \ref{thm:approximation:snfe}.

\subsection{Results on the SNFE}

\begin{thm}\label{thm:existence:uniqueness:snfe}
Under Assumption \ref{ass:f}, assume that $f$ is lower bounded by $m>0$. 
Then, for all $n\geq 1$ there exists a unique (up to modification) solution $V$ of \eqref{def:SNFE} (in the sense of Definition \ref{def:solution:snfe}) such that for all $T>0$,
\begin{equation}\label{eq:L2:control:snfe}
    \sup_{t\leq T, x\in [0,1]} \mathbb{E}\left[ |V_t(x)|^2 \right] <+\infty.
\end{equation}
\end{thm}

\begin{proof}
Without loss of generality we assume in this proof that $n=1$.

We begin with uniqueness. Suppose that $V$ and $\tilde{V}$ are two solutions and define $\Delta=V-\tilde{V}$. We have
\begin{multline*}
\Delta(t,x) = \int_{0}^{t} e^{-(t-s)}\int_{0}^{1} w(y,x) [ f(V_s(y)) - f(\tilde{V}_s(y)) ] dyds 
\\ +\int_{0}^{t} e^{-(t-s)} \int_{0}^1 w(y,x) \left[ \sqrt{f(V_s(y))} - \sqrt{f(\tilde{V}_s(y))} \right] W(ds,dy).
\end{multline*}
By Jensen and Burkhölder inequalities it follows that
\begin{multline*}
    \mathbb{E}[|\Delta(t,x)|^2] \leq 2t\|w\|^2_0 \int_{0}^{t} e^{-2(t-s)} \int_{0}^{1} \mathbb{E} \left[ \left(  f(V_s(y)) - f(\tilde{V}_s(y)) \right)^2   \right] dy ds \\
    + 2\|w\|^2_0 \int_{0}^{t} e^{-2(t-s)} \int_0^1 \E\left[ \left( \sqrt{f(V_s(y))}-\sqrt{f(\tilde{V}_s(y))} \right)^2 \right] dyds.
\end{multline*}
Then, we use the fact that $f$ and $\sqrt{f}$ are Lipschitz (since $f$ is lower bounded) and it follows that
\begin{multline*}
    \mathbb{E}[|\Delta(t,x)|^2] \leq 2t\|w\|^2_0 \|f'\|^2_0 \int_{0}^{t} e^{-2(t-s)} \int_{0}^{1} \mathbb{E} \left[ |\Delta(s,y)|^2   \right] dy ds \\
    + 2\|w\|^2_0C\|f'\|^2_0 \int_{0}^{t} e^{-2(t-s)} \int_0^1 \E\left[ |\Delta(s,y)|^2 \right] dyds.
\end{multline*}
Writing $G(t):=\sup_{x\in [0,1]} \mathbb{E}[|\Delta(t,x)|^2]$, we get
\begin{equation}\label{eq:gronwall:uniqueness:snfe}
    G(t)\leq C(t+1) \int_0^t G(s) ds,
\end{equation}
and Gronwall's lemma implies $G(t)=0$ for all $t\geq 0$ which grants the uniqueness property.

For the existence of a solution, we can proceed with Picard iteration. Let $V^{(0)}_t(x) = u_0(x)$ for all $t\geq 0$, and define iteratively on $k$,
\begin{multline}\label{eq:picard:iteration:snfe}
    V^{(k+1)}_t(x) := e^{-t} u_0(x) + \int_{0}^{t} e^{-(t-s)}\int_{0}^1 w(y,x) f(V^{(k)}_s(y))dyds\\
    +\int_{0}^{t} e^{-(t-s)} \int_{0}^1 w(y,x) \sqrt{f(V^{(k)}_s(y))} W(ds,dy).
\end{multline}
The proof of convergence of the Picard iteration is pretty classic and follows computations which are similar to the ones given above to show the uniqueness property. We give below a sketch of proof with the main steps whereas the interested reader is referred to \cite[Theorem 3.7.]{faugeras2015stochastic} for the missing computations.

Using the fact that $\sup_{x\in [0,1]} u_0(x)^2<+\infty$ (which is granted by Assumption \ref{ass:f}), one can show by induction that $\sup_{t\leq T,x\in[0,1]} \E\left[ |V^{(k)}_t(x)|^2 \right] <+\infty$ so that the stochastic integral in \eqref{eq:picard:iteration:snfe} is well defined in Walsh's sense. Then, defining $\Delta^{(k)} := V^{(k+1)} - V^{(k)}$ and $G_k(t) := \sup_{x\in [0,1]} \mathbb{E}[|\Delta^{(k)}_t(x)|^2]$ and applying induction on the same computations as the one leading to Equation \eqref{eq:gronwall:uniqueness:snfe} give
\begin{equation*}
    G_k(t)\leq C^k(t+1)^k \int_0^t \dots \int_0^{t_{k-1}} G_0(t_k) dt_kk\dots dt_1,
\end{equation*}
for $k\geq 1$ and $G_0(t)\leq C_t (1+\sup_{x\in [0,1]} u_0(x)^2)$ for some time dependent constant $C_t$. Using both inequalities above, one can show that 
$$\sup_k \sup_{t\leq T,x\in[0,1]} \E\left[ |V^{(k)}_t(x)|^2 \right]<+\infty.$$
Finally, this implies the existence of the limit $V_t(x)$ in $L^2$ and that the convergence is uniform, i.e.
\begin{equation*}
    \sup_{t\leq T,x\in[0,1]} \E\left[ |V^{(k)}_t(x) - V_t(x)|^2 \right] \to 0.
\end{equation*}
Hence, Equation \eqref{eq:L2:control:snfe} is satisfied and the uniform convergence justifies taking the limit as $k$ goes to infinity in \eqref{eq:picard:iteration:snfe} in order to prove that the limit $V$ is indeed a solution in the sense of Definition \ref{def:solution:snfe}.
\end{proof}

Furthermore, the mean-field limit is an approximation of the SFNE in the following sense.

\begin{proposition}
\label{prop:approximation:snfe}
Under the assumptions of Theorem \ref{thm:existence:uniqueness:snfe}, there exists a constant $C=C(T,w,f,u_0)$ such that the unique solution $V^n$ of \eqref{def:SNFE} satisfies
\begin{equation}\label{eq:L2:control:approximation:snfe}
    \sup_{t\leq T, x\in [0,1]} \mathbb{E}\left[ |V^n_t(x)-u(t,x)|^2 \right] \leq \frac{C}{n}.
\end{equation}
\end{proposition}
\begin{proof}
The proof is pretty similar to the proof of Theorem \ref{thm:existence:uniqueness:snfe} but here we must keep track of the index $n$. Let us denote $\Delta^n := V^n - u$ and $G^n(t):=\sup_{x\in [0,1]} \mathbb{E}[|\Delta^n_t(x)|^2]$. We have
\begin{multline*}
\Delta^n(t,x) = \int_{0}^{t} e^{-(t-s)}\int_{0}^{1} w(y,x) [ f(V^n_s(y)) - f(u(s,y)) ] dyds 
\\ + n^{-1/2} \int_{0}^{t} e^{-(t-s)} \int_{0}^1 w(y,x) \sqrt{f(V^n_s(y))}  W(ds,dy).
\end{multline*}
The same arguments (Jensen and Burkhölder inequalities) imply the existence of a constant $C$ such that
\begin{equation*}
    G^n(t)\leq C\left(n^{-1} + t \int_0^t G^n(s) ds\right),
\end{equation*}
which grants the result according to Gronwall's lemma.
\end{proof}

The following result gives the space-time regularity of the solution of the SFNE \eqref{def:SNFE}.

\begin{thm}\label{thm:regularity:snfe}
Under the assumptions of Theorem \ref{thm:existence:uniqueness:snfe}, the unique solution $V$ of \eqref{def:SNFE} is $L^p$ bounded,
\begin{equation}\label{eq:Lp:control:snfe}
    \sup_{t\leq T, x\in [0,1]} \mathbb{E}\left[ |V_t(x)|^p \right] <+\infty,
\end{equation}
for all $p\geq 2$, and there is a modification of $V$ such that $(t,x)\mapsto V_t(x)$ is $(\eta_1,\eta_2)$-Hölder continuous for any $\eta_1<1/2$ and $\eta_2<1$.
\end{thm}

\begin{proof}
Using Jensen and Burkhölder inequalities in a similar way as in the proof of Theorem \ref{thm:existence:uniqueness:snfe}, one can get, from \eqref{def:solution_SNFE}, the following inequality
\begin{equation*}
    \E\left[ |V_t(x)|^p \right] \leq 3^{p-1}\left[ |u_0(x)|^p + ||w||_0^p ||f'||_0^p \left( t^{p-1} + c_p t^{p/2-1} \right) \int_0^t \E\left[ \left( \sup_x |V_s(x)| + f(0) \right)^p \right] ds \right],
\end{equation*}
and so there is a constant $C_t$ such that $H_p(t) := \sup_{x\in [0,1]} \mathbb{E}\left[ |V_t(x)|^p \right]$ satisfies
\begin{equation*}
    H_p(t)\leq C_t \left( 1+ \sup_{x\in [0,1]} u_0(x)^p + \int_0^t H_p(s) ds \right).
\end{equation*}
Since $u_0$ is bounded, Gronwall's lemma applied to the last inequality gives \eqref{eq:Lp:control:snfe}.

For the time regularity, the same kind of computations gives for any $p$ and times $0\leq s\leq t$, the existence of a constant $C_t$ such that
\begin{equation*}
    \E\left[ |V_t(x) - V_{t'}(x)|^p \right] \leq C_t \left( 1 + H_p(t) \right) (t-t')^{p/2}.
\end{equation*}
Finally, Kolmogorov's continuity theorem gives the stated regularity (remember that $p$ can be as large as one needs).

For the spatial regularity, let us write $V_t(x) := I_1(t,x) + I_2(t,x) + I_3(t,x)$ as given in Definition \ref{def:solution:snfe} by 
\begin{equation*}
    \begin{cases}
    I_1(t,x) := e^{-t} V_0(x),\\
    I_2(t,x) := \int_{0}^{t} e^{-(t-s)}\int_{0}^1 w(y,x) f(V_s(y))dyds ,\\
    I_3(t,x) := \int_{0}^{t} e^{-(t-s)} \int_{0}^1 w(y,x) \frac{\sqrt{f(V_s(y))}}{\sqrt{n}} W(ds,dy).
    \end{cases}
\end{equation*}
For all $x$ and $x'$ in $[0,1]$, one can show that for any $p\geq 2$, and $j=1,2,3$,
\begin{equation}
    \E \left[ |I_j(t,x) - I_j(t,x')|^p\right]\leq C_t |x-x'|^p.
\end{equation}
Instead of giving the full computations, we only give the arguments here (once again, similar computations can be found in \cite[Theorem 3.10]{faugeras2015stochastic}. For $j=1$, it is a direct consequence of the Lipschitz continuity of $u_0$.  For $j=2$ and $3$, it comes from the Lipschitz continuity of $w$.

Once again, the stated regularity is obtained thanks to Kolmogorov's continuity theorem.
\end{proof}

\subsection{Approximation result}

We start this section discussing the well-posedness of both processes $(\Gamma_t(x))_{t,x}$ and $(M_t(x))_{t,x}.$  

\begin{proposition}
\label{prop:well_posedness_Gamma_M_as_function}
Fix $T>0$ and assume assumptions of Theorem \ref{thm:existence:uniqueness:snfe}.
\begin{enumerate}
    \item The process $(M_t(x))_{t\leq T,x\in [0,1]}$ given by \eqref{def:martingale_as_function}
is well-defined in $L^2$. Moreover, there is a modification of $(M_t(x))_{t,x}$ such that $(t,x)\mapsto M_t(x)$ is $(\eta_1,\eta_2)$-Hölder continuous for any $\eta_1<1/2$ and $\eta_2<1$. If we define for such a modification  
$$
M_t(\varphi)=\int_0^1\varphi(x)M_t(x)dx, \ \mbox{for} \ \varphi\in\cS \ \mbox{and} \ t\leq T,
$$
then $M=(M_t)_{t\leq T}$ is a $\cS'$-valued centered Gaussian process with covariance function defined in \eqref{eq:covariance:of:M}.
\item There exists a unique solution $(\Gamma_t(x))_{t\leq T,x\in [0,1]}$ of \eqref{def:limit_process_as_function} such that for all $T>0$,
\begin{equation*}
    \sup_{t\leq T, x\in [0,1]} \mathbb{E}\left[ |\Gamma_t(x)|^2 \right] <+\infty.
\end{equation*}
Moreover, the solution is bounded in $L^p$ for any $p\geq 2$,
$$
\sup_{t\leq T}\sup_{x\in [0,1]}\E(|\Gamma_t(x)|^p)<\infty,
$$
and there exists a modification such that $(t,x)\mapsto \Gamma_t(x)$ is $(\eta_1,\eta_2)$-Hölder continuous for any $\eta_1<1/2$ and $\eta_2<1$. 
Furthermore, if we define for such a modification
$$
\Gamma_t(\varphi)=\int_0^1\varphi(x)\Gamma_t(x)dx, \ \mbox{for} \ \varphi\in\cS \ \mbox{and} \ t\leq T,
$$
then the resulting $\cS'$-valued process $(\Gamma_t)_{t\leq T}$ is the unique solution of \eqref{def:limit_process}.
\end{enumerate}
\end{proposition}

\begin{proof}
By Proposition \ref{prop:contraction:NFE}, we have that $\sup_{t\leq T}\|u(t,\cdot)\|_0<\infty$. Using this fact and recalling that $f$ is Lipschitz continuous and $w$ is bounded, one can easily deduce that for any $t\leq T$ and $x\in [0,1],$
$$
\int_{0}^t\int_0^1e^{2\alpha s}w^2(y,x)f(u(s,y))dsdy<\infty.
$$
Then, by \cite[Theorem 3.1]{faugeras2015stochastic} the process $M_t(x)$ given by \eqref{def:martingale_as_function} is well-defined for any $t\leq T$ and $x\in [0,1].$ The proof of the regularity properties of $(M_t(x))_{t,x}$ is omitted here since it follows along the same lines as in the proof of Theorem \ref{thm:regularity:snfe}.

Now, fix $t\leq T$ and take $\varphi\in\cS$. Under assumptions of Theorem \ref{thm:existence:uniqueness:snfe} we can apply Fubini–Tonelli property for martingale integrals (see for instance \cite[Theorem 5.30]{Dalang2009AMO}) to deduce that
$$
M_t(\varphi)=\int_{0}^1\varphi(x)M_t(x)dx=\int_0^te^{\alpha s}\int_{0}^1 I(\varphi)(y)\sqrt{f(u(s,y))}W(dsdy),
$$
where $y\mapsto I(\varphi)(y)$ is defined in \eqref{eq:covariance:of:M}. By using the properties of the white noise $W$ and \cite[Theorem 3.1]{faugeras2015stochastic}, one can easily check that $M$ is indeed a $\cS'$-valued centered Gaussian process with covariance function defined in \eqref{eq:covariance:of:M}.

It remains to prove Item 2. 
The proof of existence and uniqueness of solutions of \eqref{def:limit_process_as_function} 
is similar to that of Theorem       \ref{thm:existence:uniqueness:snfe}, whereas the boundedness in $L^p$ and the regularity properties of $(\Gamma_t(x))_{t,x}$ follow along the same lines as in the proof of Theorem  \ref{thm:regularity:snfe}. For this reason the details are not given here. 
To conclude the proof of the theorem, fix $t\leq T$, take $\varphi\in S$ and observe that by the definition of $\Gamma_t(x)$ we have
$$
\Gamma_t(\varphi)=e^{-\alpha t}M_t(\varphi)+\int_{0}^1\varphi(x)\int_{0}^t
e^{-\alpha (t-s)}\int_0^1w(y,x)f'(u(s,y))\Gamma_s(y)dydsdx.$$
By Fubini Theorem, we can interchange the  order of integration of the the term in the RHS of the equation above to deduce that
$$
\Gamma_t(\varphi)=e^{-\alpha t}M_t(\varphi)+\int_{0}^t
e^{-\alpha (t-s)}\int_0^1\left(\int_{0}^1\varphi(x)w(y,x)f'(u(s,y))dx\right)\Gamma_s(y)dyds.$$
Since for any $s\geq 0$,
$$
\Gamma_s\left(\int_{0}^1\varphi(x)w(\cdot,x)f'(u(s,\cdot))dx\right)=\int_{0}^1\left(\int_{0}^1\varphi(x)w(y,x)f'(u(s,y))dx\right)\Gamma_s(y)dy,
$$
we have that $\Gamma_t(\varphi)$ solves \eqref{def:limit_process} and the result follows from uniqueness of Proposition \ref{prop:pathwise:uniqueness}.
\end{proof}

We are now ready to state and prove the approximation result: that is the control of $D^n$ defined by, for all $t\in [0,T]$ and $x\in [0,1]$,
\begin{equation}
    D^n(t,x) := Y^n_t(x) - V^n_t(x)
\end{equation}
where $Y^n_t(x)$ and $V^n_t(x)$ are respectively given by \eqref{eq:heuristics:Yn} and \eqref{def:solution_SNFE}.

\begin{thm}\label{thm:approximation:snfe}
Under assumptions of Theorem \ref{thm:existence:uniqueness:snfe}, for any $T>0$ there exists a constant $C=C(T,w,f',f^{(2)},\alpha)$ such that for all $n\geq 1$,
$$
\sup_{0\leq t\leq T}\sup_{x\in [0,1]}\E((D^n(t,x))^2)\leq \frac{C}{n^{2}}. 
$$
\end{thm}
\begin{proof}
By Jensen and Burkhölder inequalities it follows that
\begin{multline}
\label{Thm_4:ineq_1}
\E(D^n(t,x)^2)\leq \frac{2\|w\|^2_0}{n}\int_{0}^te^{-2\alpha(t-s)}\int_{0}^1\E\left(\left(\sqrt{f(u(s,y))}-\sqrt{f(V^n_s(y))}\right)^2\right)dyds\\+t\int_{0}^t e^{-2\alpha(t-s)}\int_{0}^1w^{2}(y,x)\E\left(\left|f(u(s,y))+f'(u(s,y))n^{-1/2}\Gamma_s(y)-f(V^n_s(y))\right|^2\right)dyds.  
\end{multline}
By applying Taylor approximation of order 2 and then the inequality $(a+b)^2 \leq 2(a^2+b^2)$, we obtain that
\begin{multline*}
\E\left(\left|f(u(s,y))+f'(u(s,y))n^{-1/2}\Gamma_s(y)-f(V^n_s(y))\right|^2\right)\leq 2\|f'\|^2_0\E((D^n(s,y))^2)\\+\frac{\|f^{(2)}\|_0}{2n^2}\E(\Gamma^4_s(y)).  
\end{multline*}
Now, by using first that $f$ is lower bounded by $m$ and then the fact that $f$ is Lipschitz, it follows that
$$
\E\left(\left(\sqrt{f(u(s,y))}-\sqrt{f(V^n_s(y))}\right)^2\right)\leq \frac{\|f'\|^2_0}{4m}\E\left(\left|u(s,y)-V^{n}_s(y)\right|^2\right). 
$$
Combining these last two inequalities with inequality \eqref{Thm_4:ineq_1} yields
\begin{multline*}
\E(D^n(t,x))^2\leq 
\frac{\|w\|^2_0\|f'\|_0}{2nm}\int_{0}^te^{-2\alpha(t-s)}\int_{0}^1\E\left(\left|u(s,y)-V^{n}_s(y)\right| \right)dyds\\
+2t\|w\|^2_0\|f'\|_0^2\int_{0}^t e^{-2\alpha(t-s)}\int_{0}^1\E((D^n(s,y))^2)dyds\\+
\frac{t\|w\|^2_0\|f^{(2)}\|_0^2}{2n^2}\int_{0}^t e^{-2\alpha(t-s)}\int_{0}^1\E(\Gamma^4_s(y))dyds.
\end{multline*}
Let $H(t)=\sup_{x\in [0,1]}\E((D^n(t,x))^2)$. By Proposition \ref{prop:approximation:snfe} and Theorem \ref{prop:well_posedness_Gamma_M_as_function}, it follows that for all $t\leq T$,
$$
H(t)\leq \frac{C_1}{n^{2}}+C_2\int_{0}^t H(s)ds,
$$
for positive constants $C_1$ and $C_2$ depending only on $T,w,f',f^{(2)}$ and $\alpha$. 
Proposition \ref{prop:contraction:NFE}
together with Theorems \ref{thm:existence:uniqueness:snfe} and \ref{prop:well_posedness_Gamma_M_as_function} imply that the function $t\mapsto H(t)$ is locally bounded, so that the result follows from Gronwall inequality.
\end{proof}

\section{Acknowledgments}

This research has been conducted as part of FAPESP project {\em Research, Innovation and
Dissemination Center for Neuromathematics} (grant 2013/07699-0). We also acknowledge the support of CNRS under the grant PEPS JCJC \emph{MaNHawkes}.

\newpage

\appendix
\section{Lemmas}
\label{App:lemmas}

\begin{lemma}
\label{lem:Ck:bound}
Let $f$ and $g$ be in $\cC^k_b(\bR,\bR)$. Then $fg$ and $f\circ g$ are in $\cC^k_b(\bR,\bR)$ and there exists $C>0$ such that
\begin{equation*}
    ||fg||_k \leq C ||f||_k ||g||_k \text{ and } ||f\circ g||_k \leq C ||f||_k (1+||g||_{k}^{k}).
\end{equation*}
\end{lemma}
\begin{proof}
From Leibniz rule, it is clear that $||fg||_k \leq C ||f||_k ||g||_k$. For the second statement, one can proceed by induction on $k$ using $||f\circ g||_k \leq ||f||_0 + ||g'\, f'\circ g||_{k-1}$ and the first statement.
\end{proof}

\begin{lemma}
\label{lemma:truncation}
For $T>0$ and integers $n,k\geq 1$, let $\tau_k=\inf\{0\leq t\leq T : \max_{i\in [n]}|\eta^i_t|\geq k\}$.
If $f$ is Lipschitz continuous, both $u_0$ and $w_0$ are bounded, and $u(t,x)\in \cC([0,T],\cC[0,1])$, then $\tau_k\to T$ almost surely as $k\to\infty$.
\end{lemma}
\begin{proof}
By Markov inequality, we have
\begin{align*}
\P(\tau_k<t)&=\P(\sup_{s\leq t}\max_{i\in [n]}|\eta^i_s|\geq k)\\
&\leq \frac{1}{k}\E(\sup_{s\leq t}\max_{i\in [n]}|\eta^i_s|)\\
&\leq \frac{1}{k}\E(\sup_{s\leq T}\max_{i\in [n]}|\eta^i_s|)
\end{align*}
For each $s\geq 0$ and $i\in [n]$, since $\eta^i_s=\sqrt{n}(U^i_s-u(s,x_i))$ and  by assumption $M_T=\sup_{s\leq T} \|u(s,\cdot)\|_0<\infty$, it follows that
\begin{align*}
|\eta^i_s|\leq \sqrt{n}(|U^i_s|+M_T).    
\end{align*}
Now, one can check that for all $s\leq T$ and $i\in [n]$,
$$
|U^i_s|\leq \|u_0\|_0+\|w\|_0\frac{1}{n}\sum_{j=1}^nN^j_T,
$$
so that \cite[Proposition 3]{chevallier2018spatial} ensures that
$$
\E\left(\sup_{s\leq T}\max_{i\in [n]}|U^i_s|\right)\leq \|u_0\|_0+\|w\|_0\frac{1}{n}\sum_{j=1}^n\E(N^j_T)<\infty.
$$
Collecting the estimates above, we deduce that for some positive constant $C=C(T,w,u,n)$, it holds
$\E\left(\sup_{t\leq T}\max_{i\in [n]}|\eta^i_t|\right)<C$, implying if $\tau=\lim_k \tau_k$, then for all $t\leq T,$
$$
\P(\tau<t)=\lim_{k\to\infty}\P(\tau_k<t)=0,
$$
and the result follows.
\end{proof}

\section{Fréchet spaces}
\label{app:frechet}

Here are gathered some technical definitions and results about semi-normed spaces and Fréchet spaces in particular. Most of what appears here is taken from \cite{simon2017banach}.

In the following, let $E$ denote a separated semi-normed space equipped with the family of semi-norms $\{N_\nu, \nu \in \cN_E\}$ and $e$ denote a generic element of $E$. The family is said to be \emph{filtering} if for all finite subset $N\subset \cN_E$, there exists $\mu$ in $\cN_E$ such that, for all $e\in E$,
\begin{equation*}
    \sup_{\nu\in N} ||u||_\nu \leq ||u||_\mu.
\end{equation*}
The dual space $E'$ is the space of continuous linear forms $\xi$ on $E$. If the family of norms is filtering then there is a simple characterization of $E'$:
\begin{equation*}
    \xi\in E' \quad \Leftrightarrow \quad \exists \nu\in \cN_E,\, c>0,\ \sup_{e\in E} \frac{|\xi(e)|}{||e||_\nu} \leq c.
\end{equation*}

\begin{definition}
Any $B\subset E$ is called bounded if for every $\nu\in \cN_E$, $\sup_{e\in B} ||e||_\nu <+\infty$.
\end{definition}
In this paper, we endow $E'$ with the \emph{strong topology} defined by the family of semi-norms indexed by the bounded sets $B$ of $E$,
\begin{equation*}
    ||\xi||_{B} := \sup_{e\in B} |\xi(e)|.
\end{equation*}
Hence, $\xi_n\to \xi$ in $E'$ is equivalent to $||\xi_n - \xi||_B \to 0$ for every bounded set $B$. In particular, $\xi_n \to \xi$ implies $\xi_n(e) - \xi(e)$ for all $e$ in $E$.

\begin{definition}
A Fréchet space is any sequentially complete metrizable semi-normed space.
\end{definition}
All through the paper, the Fréchet space of interest is $\cS$ equipped with its natural filtering family of semi-norms. Hence, for instance, the space $\cC([0,T],\cS')$ is understood as the space of continuous functions $\gamma$ from $[0,T]$ to $\cS'$ equipped with the strong topology. Its topology is given by the projective limit topology of $\{\xi\mapsto \sup_{0\leq t\leq T} ||\xi_t||_{B}, \, B \text{ bounded set of }E\}$. In particular, $\gamma\in \cC([0,T],\cS')$ implies $\gamma(\varphi)\in \cC([0,T],\bR)$ for all $\varphi$ in $\cS$.
The construction of $\cD([0,T],\cS')$ follows the same idea where the sup norm over $t\in [0,T]$ is replaced by Skorokhod metric.

\newpage

\bibliographystyle{abbrv}
\bibliography{references}

\begin{thebibliography}{10}

\bibitem{Amari:77}
S.-i. Amari.
\newblock Dynamics of pattern formation in lateral-inhibition type neural
  fields.
\newblock {\em Biological Cybernetics}, 27(2):77--87, Jun 1977.

\bibitem{Baladron:12}
J.~Baladron, D.~Fasoli, O.~Faugeras, and J.~Touboul.
\newblock Mean-field description and propagation of chaos in networks of
  hodgkin-huxley and fitzhugh-nagumo neurons.
\newblock {\em The Journal of Mathematical Neuroscience}, 2(1):10, 2012.

\bibitem{Billing_Convergence}
P.~Billingsley.
\newblock {\em {Convergence of probability measures}}.
\newblock {Wiley Series in Probability and Statistics: Probability and
  Statistics}. John Wiley \& Sons Inc., New York, second edition, 1999.
\newblock A Wiley-Interscience Publication.

\bibitem{bossy2015clarification}
M.~Bossy, O.~Faugeras, and D.~Talay.
\newblock Clarification and complement to ``mean-field description and
  propagation of chaos in networks of hodgkin--huxley and fitzhugh--nagumo
  neurons''.
\newblock {\em The Journal of Mathematical Neuroscience (JMN)}, 5(1):1, 2015.

\bibitem{Braun1977vlasov}
W.~Braun and K.~Hepp.
\newblock {The Vlasov dynamics and its fluctuations in the 1/N limit of
  interacting classical particles}.
\newblock {\em Communications in Mathematical Physics}, 56(2):101--113, Jun
  1977.

\bibitem{bressloff2011spatiotemporal}
P.~C. Bressloff.
\newblock Spatiotemporal dynamics of continuum neural fields.
\newblock {\em Journal of Physics A: Mathematical and Theoretical},
  45(3):033001, 2011.

\bibitem{Witten_et_al:17}
S.~Chen, A.~Shojaie, E.~Shea-Brown, and D.~Witten.
\newblock The multivariate hawkes process in high dimensions: Beyond mutual
  excitation.
\newblock {\em ArXiv}, 2017.

\bibitem{chevallier2017fluctuations}
J.~Chevallier.
\newblock Fluctuations for mean-field interacting age-dependent hawkes
  processes.
\newblock {\em Electron. J. Probab.}, 22:1--49, 2017.

\bibitem{chevallier2015microscopic}
J.~Chevallier, M.~J. C{\'a}ceres, M.~Doumic, and P.~Reynaud-Bouret.
\newblock Microscopic approach of a time elapsed neural model.
\newblock {\em Mathematical Models and Methods in Applied Sciences},
  25(14):2669--2719, 2015.

\bibitem{chevallier2018spatial}
J.~Chevallier, A.~Duarte, E.~L{\"o}cherbach, and G.~Ost.
\newblock {Mean field limits for nonlinear spatially extended Hawkes processes
  with exponential memory kernels}.
\newblock {\em Stochastic Processes and their Applications}, 2018.

\bibitem{Chornoboy:88}
E.~Chornoboy, L.~Schramm, and A.~Karr.
\newblock {Maximum likelihood identification of neural point process systems}.
\newblock {\em Biological Cybernetics}, 59(4-5):265--275, 1988.

\bibitem{Dalang2009AMO}
R.~C. Dalang, D.~Khoshnevisan, C.~Mueller, D.~Nualart, and Y.~Xiao.
\newblock {\em A minicourse on stochastic partial differential equations},
  volume 1962.
\newblock Springer, 2009.

\bibitem{DAWSON1991law}
D.~Dawson and X.~Zheng.
\newblock Law of large numbers and central limit theorem for unbounded jump
  mean-field models.
\newblock {\em Advances in Applied Mathematics}, 12(3):293 -- 326, 1991.

\bibitem{delarue2019masterequation}
F.~Delarue, D.~Lacker, and K.~Ramanan.
\newblock From the master equation to mean field game limit theory: a central
  limit theorem.
\newblock {\em Electron. J. Probab.}, 24:54 pp., 2019.

\bibitem{SusanneEva:17}
S.~Ditlevsen and E.~L\"ocherbach.
\newblock Multi-class oscillating systems of interacting neurons.
\newblock {\em Stoch. Proc. Appl}, 127:1840--1869, 2017.

\bibitem{ermentrout1998neural}
B.~Ermentrout.
\newblock Neural networks as spatio-temporal pattern-forming systems.
\newblock {\em Reports on progress in physics}, 61(4):353, 1998.

\bibitem{faugeras2015stochastic}
O.~Faugeras and J.~Inglis.
\newblock Stochastic neural field equations: a rigorous footing.
\newblock {\em Journal of mathematical biology}, 71(2):259--300, 2015.

\bibitem{gill_1997}
R.~D. Gill, N.~Keiding, and P.~K. Andersen.
\newblock {\em Statistical models based on counting processes}.
\newblock Springer, 1997.

\bibitem{hansen2015lasso}
N.~R. Hansen, P.~Reynaud-Bouret, and V.~Rivoirard.
\newblock Lasso and probabilistic inequalities for multivariate point
  processes.
\newblock {\em Bernoulli}, 21(1):83--143, 2015.

\bibitem{hawkes_1971}
A.~G. Hawkes.
\newblock Spectra of some self-exciting and mutually exciting point processes.
\newblock {\em Biometrika}, 58(1):83--90, 1971.

\bibitem{HITSUDA1986tightness}
M.~Hitsuda and I.~Mitoma.
\newblock Tightness problem and stochastic evolution equation arising from
  fluctuation phenomena for interacting diffusions.
\newblock {\em Journal of Multivariate Analysis}, 19(2):311 -- 328, 1986.

\bibitem{hodara_locherbach:17}
P.~Hodara and E.~Löcherbach.
\newblock Hawkes processes with variable length memory and an infinite number
  of components.
\newblock {\em Advances in Applied Probability}, 49(1):84–107, 2017.

\bibitem{jacod1979calcul}
J.~Jacod.
\newblock {\em Calcul stochastique et problemes de martingales}.
\newblock Springer, 1979.

\bibitem{Jacod_2003}
J.~Jacod and A.~N. Shiryaev.
\newblock {\em Limit theorems for stochastic processes}, volume 288 of {\em
  Grundlehren der Mathematischen Wissenschaften [Fundamental Principles of
  Mathematical Sciences]}.
\newblock Springer-Verlag, Berlin, second edition, 2003.

\bibitem{Johnson:96}
D.~H. Johnson.
\newblock Point process models of single-neuron discharges.
\newblock {\em Journal of Computational Neuroscience}, 3(4):275--299, Dec 1996.

\bibitem{kallianpur2013stochastic}
G.~Kallianpur.
\newblock {\em Stochastic filtering theory}, volume~13.
\newblock Springer Science \& Business Media, 2013.

\bibitem{kurtz2004stochastic}
T.~G. Kurtz and J.~Xiong.
\newblock A stochastic evolution equation arising from the fluctuations of a
  class of interacting particle systems.
\newblock {\em Commun. Math. Sci.}, 2(3):325--358, 09 2004.

\bibitem{luccon2016transition}
E.~Lu{\c{c}}on and W.~Stannat.
\newblock Transition from gaussian to non-gaussian fluctuations for mean-field
  diffusions in spatial interaction.
\newblock {\em The Annals of Applied Probability}, 26(6):3840--3909, 2016.

\bibitem{Meleard_96}
S.~M{{\'e}}l{{\'e}}ard.
\newblock Asymptotic behaviour of some interacting particle systems;
  {M}c{K}ean-{V}lasov and {B}oltzmann models.
\newblock In {\em Probabilistic models for nonlinear partial differential
  equations ({M}ontecatini {T}erme, 1995)}, volume 1627 of {\em Lecture Notes
  in Math.}, pages 42--95. Springer, Berlin, 1996.

\bibitem{Mitoma_89}
I.~Mitoma.
\newblock Tightness of probabilities on c([ 0, 1 ]; y') and d([ 0, 1 ]; y').
\newblock {\em The Annals of Probability}, 11(4):989--999, 1983.

\bibitem{Pernice:11}
V.~Pernice, B.~Staude, S.~Cardanobile, and S.~Rotter.
\newblock How structure determines correlations in neuronal networks.
\newblock {\em PLOS Computational Biology}, 7(5):1--14, 05 2011.

\bibitem{Rebolledo_80}
R.~Rebolledo.
\newblock Central limit theorems for local martingales.
\newblock {\em Z. Wahrsch. Verw. Gebiete}, 51(3):269--286, 1980.

\bibitem{Revuz_1999}
D.~Revuz and M.~Yor.
\newblock {\em Continuous Martingales and Brownian Motion (Grundlehren der
  mathematischen Wissenschaften)}.
\newblock Springer-Verlag, 3rd edition, 1999.

\bibitem{Bouret:14}
P.~Reynaud-Bouret, V.~Rivoirard, F.~Grammont, and C.~Tuleau-Malot.
\newblock Goodness-of-fit tests and nonparametric adaptive estimation for spike
  train analysis.
\newblock {\em The Journal of Mathematical Neuroscience}, 4(1):3, Apr 2014.

\bibitem{simon2017banach}
J.~Simon.
\newblock {\em Banach, Fr{\'e}chet, Hilbert and Neumann Spaces}.
\newblock John Wiley \& Sons, 2017.

\bibitem{tran2006modeles}
V.~C. Tran.
\newblock {\em Mod{\`e}les particulaires stochastiques pour des probl{\`e}mes
  d'{\'e}volution adaptative et pour l'approximation de solutions
  statistiques}.
\newblock PhD thesis, Universit{\'e} de Nanterre-Paris X, 2006.

\bibitem{walsh1986introduction}
J.~B. Walsh.
\newblock An introduction to stochastic partial differential equations.
\newblock In {\em {\'E}cole d'{\'E}t{\'e} de Probabilit{\'e}s de Saint Flour
  XIV-1984}, pages 265--439. Springer, 1986.

\bibitem{wilson_Conwan:72}
H.~R. Wilson and J.~D. Cowan.
\newblock Excitatory and inhibitory interactions in localized populations of
  model neurons.
\newblock {\em Biophysical Journal}, 12:1--24, 1972.

\bibitem{Wilson_Conwan:73}
H.~R. Wilson and J.~D. Cowan.
\newblock A mathematical theory of the functional dynamics of cortical and
  thalamic nervous tissue.
\newblock {\em Kybernetik}, 13(2):55--80, Sep 1973.

\end{thebibliography}

\end{document}